\documentclass[12pt,letterpaper]{amsart}
\pdfoutput=1 
\usepackage{etex}
\usepackage[latin1]{inputenc}
\usepackage{multirow}
\usepackage{amsmath}

\usepackage{amsfonts}
\usepackage{amsthm}
\usepackage{array} 
\usepackage{amsthm}
\usepackage{tikz-cd}	
\usetikzlibrary{arrows}
\usetikzlibrary{patterns} 

\usepackage{amssymb}
\usepackage{amscd, amsfonts}
\usepackage[dvips]{epsfig}
\usepackage{verbatim}
\usepackage[usenames,dvipsnames]{pstricks}
\usepackage{epsfig}
\usepackage{pst-grad} 
\usepackage{pst-plot} 
\usepackage{pdftricks}
\usepackage{pstricks}
\usepackage{epsf}

\usepackage{multirow}
\usepackage{amssymb}
\usepackage{verbatim} 
\usepackage{graphicx}
\usepackage{xypic}
\usepackage[all]{xy}
\usepackage{lscape,soul}
\usepackage{hyperref}

\numberwithin{equation}{subsection}
\newtheorem{theorem}[subsection]{Theorem}

\newtheorem{proposition}[subsection]{Proposition}

\newtheorem{lemma}[subsection]{Lemma}
\newtheorem{claim}[subsection]{Claim}

\newtheorem{corollary}[subsection]{Corollary}
\newtheorem{definition}[subsection]{Definition}
 
\theoremstyle{definition}
\newtheorem{remark}[subsection]{Remark}
\newtheorem{example}[subsection]{Example}
\newtheorem{notation}[subsection]{Notation}

\newcommand{\bC}{\mathbb{C}}

\newcommand{\spec}{\mathrm{Spec}\;}

\newcommand{\PP}{\mathbb{P}}
\newcommand{\bP}{\mathbb{P}}
\newcommand{\bA}{\mathbb{A}}

\newcommand{\SL}{\textrm{SL}}

\newcommand{\calO}{\mathcal{O}}

\newcommand{\calU}{\mathcal{U}}

\newcommand{\iso}{\cong}

\usepackage{graphicx}
\newcommand{\sslash}{\mathbin{/\mkern-6mu/}}

\newcommand{\R}{R}
\newcommand{\bl}{\textrm{Bl}_x(\bP^2)}

\newcommand{\ChowQ}{/\hspace{-1.2mm}/_{Ch}}

\begin{document}
\title{
A generic slice of the moduli space of line arrangements
}
\author[Ascher]{Kenneth Ascher}
\email{kascher@mit.edu}
\author[Gallardo]{Patricio Gallardo}
\email{gallardo@uga.edu}
\address{ }
\email{}
\maketitle{}

\begin{abstract}
We study the compactification of the locus parametrizing lines having a fixed intersection with a given
line, inside the moduli space of line
arrangements in the projective plane constructed for weight one by Hacking-Keel-Tevelev and Alexeev for general weights.  We show that this space is smooth, with
normal crossing boundary, and that it has a morphism to the moduli space of
marked rational curves which can be understood as a natural
continuation of the blow up construction of Kapranov.  In addition, we prove that our space is isomorphic to a closed subvariety inside a non-reductive Chow quotient. 
\end{abstract}
\section{Introduction}

The compact moduli space of weighted hyperplane arrangements in $\bP^2$ is a  higher dimensional generalization of $\overline{M}_{0,n}$, and has a main component parameterizing  equivalence classes of $n$ weighted lines in $\PP^2$ and their log canonical degenerations.  The moduli space 
$\overline{M}_{\vec \beta}(\PP^2,n)$ was constructed for lines of weight one by Hacking-Keel-Tevelev \cite{hkt}, and for more general weights $\vec \beta$ as a generalization of the weighted Hassett spaces, by Alexeev \cite{alexeev2013moduli}. The space is expected to satisfy Murphy's law-- it can be arbitrarily singular, and  can contain many irreducible components.   The goal of this paper is to describe a naturally appearing locus inside this moduli space which has perhaps unexpected properties -- it is smooth with normal crossings boundary. 

Given an arrangement of $(n+1)$ labeled lines in $\PP^2$, there is a natural restriction morphism: label the line $l_{n+1}$ as $l_A$, and obtain an arrangement of $n$ labeled points on $l_A \cong \bP^1$, by intersecting the other $n$ lines with $l_A$. The restriction morphism  induces a morphism $M_{(\vec w,1)}(\bP^2, n+1) \to M_{0,\vec w}$ that has rational fibers of dimension $n-3$ (see Lemma \ref{dim}).  Given a generic point $q \in M_{0,\vec w}$, 
we study  the closure, which we denote by $\R_{\vec w}(q)$, in $ \overline M_{(\vec w,1)}(\bP^2, n+1) $ of the fiber  of $M_{(\vec w,1)}(\bP^2, n+1) \to M_{0,\vec w}$  over $q$ (see Definition \ref{def:Rn}).

In other words, $\R_{\vec w}(q)$ compactifies the locus parametrizing equivalence classes of $n+1$ labeled lines having a fixed intersection with the line $l_A$. Our first theorem characterizes $\R_{\vec w}(q)$.

\begin{theorem}[see Theorem \ref{thm:iso} and Theorem \ref{main}] \label{main1}
For weights $\vec w$ in the set of admissible weights $\mathcal D^R_{n}$
(see Definition \ref{def:adw}) and generic choice of $q \in M_{0, \vec w}$, the locus $\R_{\vec w}(q)$ is smooth with  normal crossings boundary and there are birational morphisms
$$
\centerline{\xymatrix{
\R_{ \vec w}(q)  \ar@{->}[r]^{\pi_2} &  \overline{M}_{0,\vec w}  \ar@{->}[r]^{\pi_1}
&
 \mathbb{P}^{n-3}.
}}
$$
\end{theorem}
 By results of Kapranov \cite{kap} 
the morphism  $\overline{M}_{0,n}  \to  \mathbb{P}^{n-3} $
factors into the sequence of following morphisms: 
 The blow up of $(n-1)$ points $q_i \in \PP^{n-3}$ which are in general position; the blow up of the strict transforms of the $\PP^1$'s spanned by pairs of the points $q_i$, and so forth.
For  $\vec w=(1, \ldots, 1)$,  the morphism $\pi_2$ factors in a similar fashion. 

\begin{corollary}(see Corollary \ref{genw})
The morphism  $ \R_{1^n}(q) \to  \overline{M}_{0,n}$
factors into the sequence of following morphisms: The blow up of a point $q_n$
 in the interior of  $\overline{M}_{0,n}$; the blow up of the strict transforms of the $\PP^1$s spanned by pairs $\{ q_i, q_n \}$;
the blow up of the strict transforms of the $\PP^2$s spanned by triples $\{ q_i, q_j, q_n \}$,
 and so forth.
\end{corollary}
In contrast to $\overline{M}_{0,n}$, the centers used to construct $\R_{1^n}(q)$ are not projectively equivalent to each other. As a result, $\R_{1^n}(q)$ depends on the choice of $q_n$, and in general different $q_n$ yields non-isomorphic spaces.  Moreover we show the following.
\begin{theorem} \label{exclu}
For a generic choice of $q$ and $n\geq 5$,  there do not exist weights $\vec w$ such that  $\R_{\vec w}(q) \cong \overline M_{0,n}$. 
\end{theorem}

The objects parametrized by  $\overline{M}_{\vec \beta}(\PP^2,n+1)$ are called stable hyperplane arrangements, or \emph{shas} (see \cite[Def 5.3.1]{alexeev2013moduli}), and they are stable pairs in the sense of the Minimal Model Program (see \cite[Thm 5.3.2]{alexeev2013moduli}). The  shas parametrized by $\R_{\vec w}(q)$ are described in Section \ref{setUp}. In particular, our setting restricts the possible singularities that appear in our shas (see Remark \ref{rmk:singularities} and Proposition \ref{sing}) (see Figure \ref{fig:fig1}).
\begin{figure}[!htb]\label{fig:fig1}
\minipage{0.23\textwidth}
  \includegraphics[width=\linewidth]{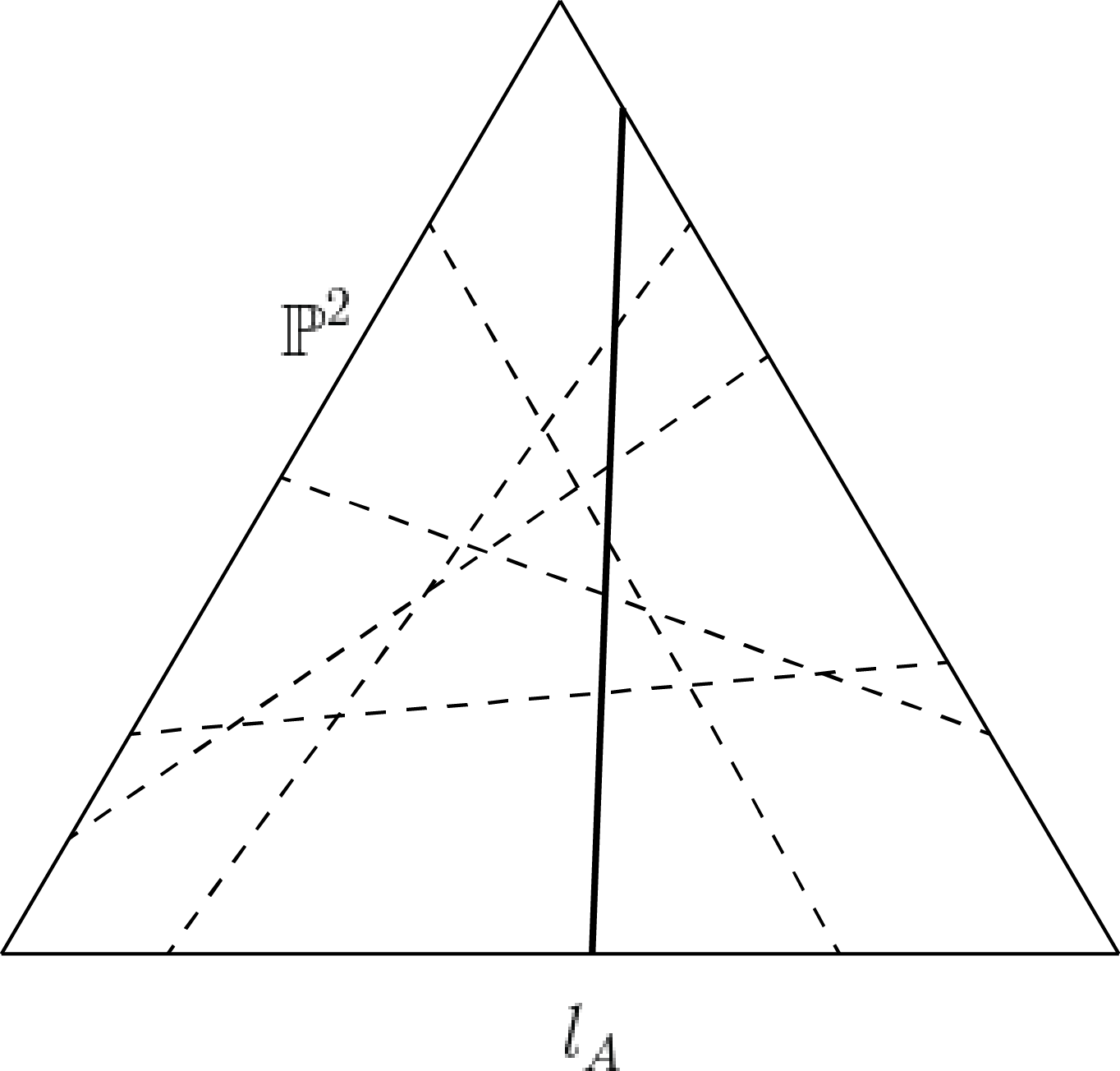}
\endminipage\hfill
\minipage{0.23\textwidth}
  \includegraphics[width=\linewidth]{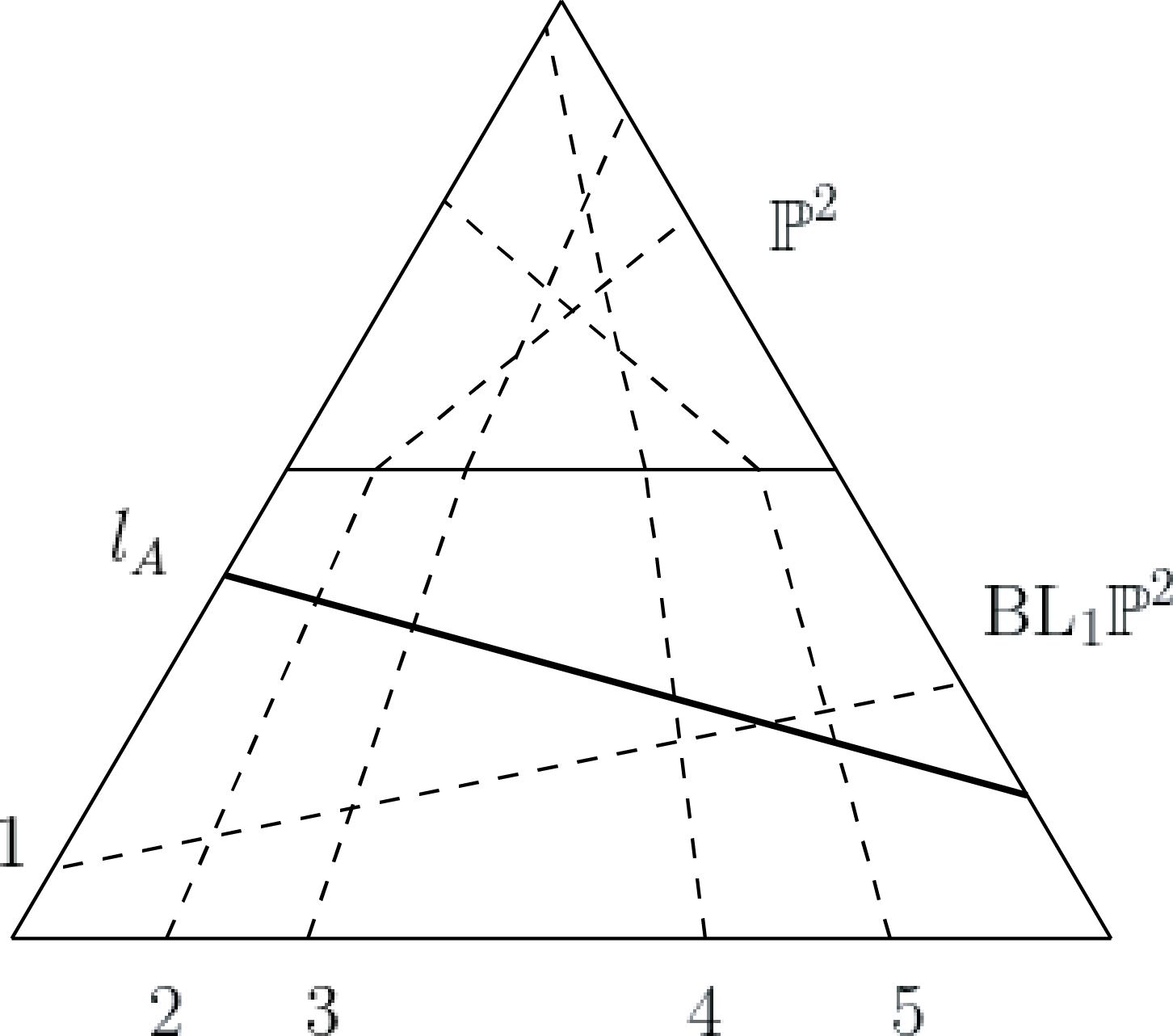}
\endminipage\hfill
\minipage{0.23\textwidth}%
  \includegraphics[width=\linewidth]{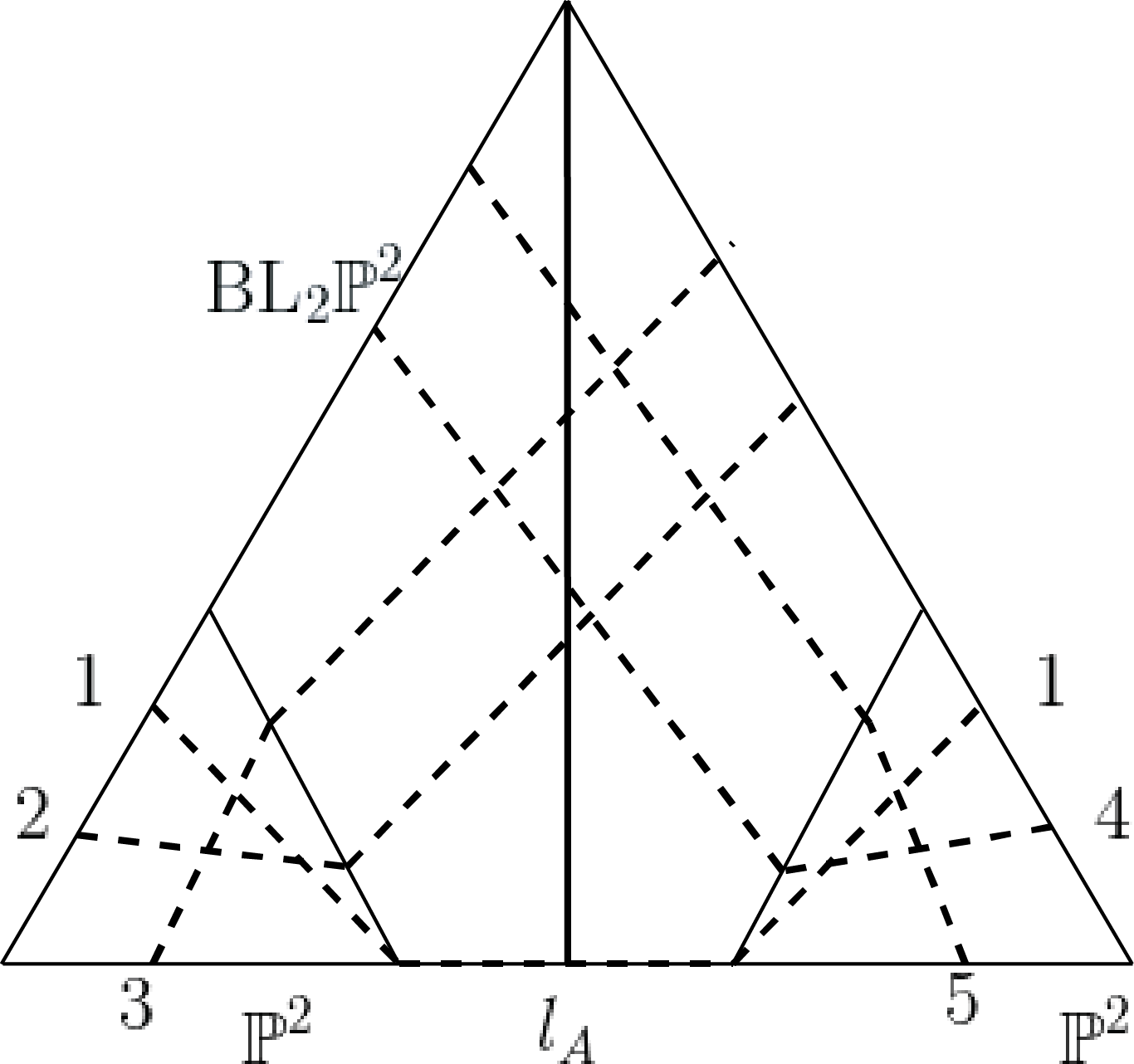}
\endminipage
\caption{Examples of generic and non-generic  shas parametrized by $\R_{1^5}(q)$}\label{examples}
\end{figure}

Our next main result is that the locus $\R_{1^n}(q)$ is the normalization of a non-reductive Chow quotient.  In particular, our result fits into a library of examples (see \cite{pn},  \cite{thaddeus1999complete}, \cite{hu2005topological},  
\cite{noahrnc} and \cite{kapranov1991quotients}) where Chow quotients are used to study the geometry of moduli spaces.   The following outline generalizes the construction of Kapranov \cite{chow} in the setting of $\R_{1^n}(q)$  (see Remark \ref{rmk:kapranov}):
Given the collection of  $n$ points $p_i$  in the dual projective space $ \hat \bP^2$ such that the point $p_i$ is dual to the line $l_i$, we consider the locus, in an appropriate Chow variety, that parametrizes  the cycles associated to the orbits $\overline{ G \cdot (p_1, \ldots, p_n)}$ where $G \subset SL(3,\mathbb{C})$ is the group that \emph{fixes the intersection} of the associated lines $l_i$ with $l_A$. By normalizing the closure of this locus in the Chow variety, we recover $R_{1^n}(q)$   (see Section \ref{sec:chow}).

\begin{theorem}[see Theorem \ref{chowt}] \label{ChowThm}
For a generic choice of $q$,  the space $\R_{1^n}(q)$ is isomorphic to the normalization of a closed subvariety of the Chow quotient $(\hat \bP^2)^n \ChowQ G$ where $G \subset \SL(3,\bC)$ is the group 
fixing the line $l_A$ pointwise.
 \end{theorem}
 
\subsection{Method of proof of Theorem \ref{main1}}
We give an outline of our proof that $\R_{\vec w}(q)$ is smooth with normal crossings boundary. The overall strategy is to prove that $\R_{\vec w}(q)$ is isomorphic to a wonderful compactification, which is smooth with normal crossings boundary by definition (see Theorem \ref{thmLi}).\\

We first construct our space with smallest admissible weights $\vec w_0$, show that $R_{\vec w_0} \cong \bP^{n-3}$ (see Lemma \ref{base}), and construct a family over $R_{\vec w_0}$ (see Lemma  \ref{lemma:universal}). In Section \ref{sec:wonderful} we construct the wonderful compactification $Bl_{\vec w} R_{\vec w_0}$, and in Lemma \ref{lem:universaltotal} we construct a family of shas over the wonderful compactification. Using this family, we obtain a finite birational (i.e. normalization) morphism from the wonderful compactification to our space: $Bl_{\vec w} R_{\vec w_0} \to R_{\vec w}$. We prove normality of $R_{\vec w}$ in Theorem \ref{thm:iso}, which implies that $R_{\vec w} \cong Bl_{\vec w} R_{\vec w_0}$ by Zariski's main theorem. Finally, we note that the key lemma required to prove normality of $R_{\vec w}$ is Lemma \ref{fibers}.

\section*{Acknowledgements} We would like to thank Dan Abramovich, Valery Alexeev, Dori Bejleri,  Noah 
Giansiracusa, Paul Hacking, Brendan Hassett, Sean Keel, Steffen Marcus, and Dhruv Ranganathan  for insightful discussions.  We thank the referees for their suggestions which greatly helped improve our work. K.A. especially thanks Steffen Marcus for help understanding deformation theory leading to the proof of Lemma \ref{fibers}. Research of P. G. is supported in part by funds from   NSF grant DMS-1344994 of the RTG in Algebra, Algebraic Geometry, and Number Theory, at the University of Georgia. K. A. is supported in part by funds from NSF grant DMS-1500525 grant, NSF grant DMS-1162367, and an NSF postdoctoral fellowship.
 \section{Definition and basic properties}\label{setUp}  
We work only over $\mathbb C$ for convenience. We begin with the necessary background on the moduli space $\overline{M}_{\vec w}(\bP^2,n+1)$, see \cite{hkt} and \cite{alexeev2013moduli} for a full exposition. 

Configurations of $(n+1)$ labeled lines $(l_1, ..., l_{n+1})$ in $\bP^2$ up to projective equivalence are parametrized by the open moduli space $M(\bP^2,n+1)$, which has a family of geometric compactifications $\overline{M}_{\vec  \beta}(\bP^2, n+1)$ depending on a weight vector 
$\vec \beta :=(\beta_1, \ldots, \beta_{n+1})$ 
(see \cite[Theorem 5.4.2]{alexeev2013moduli}). \\

The weight domain of possible weights $\vec \beta$ is

\begin{align}\label{eq:wh}
\mathcal D(3,n+1)
=\left\{
\vec \beta \in \mathbb Q^{n+1} \; 
\bigg\vert
\; \sum_{i=1}^{n+1} \beta_i > 3, \; 0 < \beta_i \leq 1
\right\}
\end{align}
In general these compactifications are \emph{not} irreducible. However, they do contain a main irreducible component parameterizing  stable pairs in the sense of MMP
 $
\left(
X, \sum_{k=1}^{n+1} \beta_{k}l_k
\right)
$
appearing as degenerations of  the $(n+1)$ lines
in $ \bP^2$. 

\begin{definition}\label{def:sha}
The stable pairs 
$
(X, D):=
\left(
X, \sum_{i=1}^{n+1} \beta_{k} l_k
\right)$  parametrized by $\overline{M}_{\vec \beta}  (\PP^{2},n+1)$ are called
\textbf{shas} of weight $\vec \beta$ or just shas if the weight $\vec \beta$ is clear from the context. 
\end{definition}

\begin{notation}\label{notate}
Let $I \subset \{1, 2, ..., n\}$ be an index set.  A sha $(X,D)$ has a multiple point $p(I)$ if there exists a component $X_i$ of $X$ and divisors  
$\{l_i = D|_{X_i} \mid  i \in I \}$ 
such that the divisors $l_i$ are concurrent at  a point $p(I) \in X_i$.
 \end{notation}

\begin{remark}\label{rmk:weights}The admissible singularities of the divisors $D$ in the sha $(X,D)$ depend completely on the weights $\vec  \beta$.  Indeed, we cannot have coincident lines $\{l_i \; | \; i \in I \}$ with weight 
$\sum_{i \in I} \beta_i>1$ or multiple points $p(I)$ defined by the concurrent lines 
$\{ l_i \; | \; i \in I \}$
with total weight $\sum_{i \in I} \beta_i >2$. \end{remark}

\begin{definition}\label{def:ParOrdW}
Let $\vec \beta$ and $\vec \alpha$ be two weights vector in $\mathcal D(3,n+1)$.
We say that $\vec \beta \geq \vec \alpha$  if $\beta_i \geq \alpha_i$ for all $i$.
\end{definition}

 As in the Hassett spaces 
$\overline M_{0, \vec w}$, the shas parametrized by $ \overline{M}_{(\vec w,1)}(\mathbb{P}^2,n+1)$ depend solely on the weights $\vec w$, and the weight domain admits a wall and chamber decomposition.
\begin{theorem}
(see \cite[Thm 5.5.2]{alexeev2013moduli})
\label{thm:wallchamber}
The domain $\mathcal D(3,n+1)$ is divided 
into finitely many walls and chambers. There are two types of walls:
\begin{align}\label{wall2}
W(I):=\left( \sum_{i \in I}\beta_i -2  =0 \right), 
& &
\widetilde{W}(I):=\left( \sum_{i \in I}\beta_i -1  =0 \right). \qquad{}
\end{align}
for all  $I \subset \{1, \ldots, n+1\}$, $2 \leq |I| \leq (n-1)$. 
Moreover,
\begin{enumerate}
\item 
if $\vec \beta$ and $\vec \alpha$ lie in the same chamber, then the weighted moduli spaces and their  families of shas are the same. 
\item 
If
 $\vec \beta$ is in the closure of the chamber containing $\vec \alpha$, then there exists a contraction 
$$
\overline{M}_{\vec \alpha}  (\PP^{2},n+1)
\to 
\overline{M}_{\vec \beta}  (\PP^{2},n+1)
$$
\item 
Further, if $\vec \beta$ is in the closure of the chamber containing $\vec \alpha$ and
$\alpha \leq \vec \beta$ then 
$$
\overline{M}_{\vec \alpha}  (\PP^{2},n+1)
= 
\overline{M}_{\vec \beta}  (\PP^{2},n+1).
$$
\end{enumerate}
\end{theorem}

\begin{remark} Recall from Remark \ref{rmk:weights} that there are two types of singularities appearing in shas. In this setting, the walls $W(I)$ correspond
 to multiple points $p(I)$, and the walls $\widetilde{W}(I)$ correspond to coincident lines. \end{remark}

\section{Definition of $\R_{\vec w}(q)$}\label{sec:definition}
To construct $\R_{\vec w}(q)$, we consider arrangements of $n+1$ labeled lines in $\bP^2$, and we label the $(n+1)^{\textrm{st}}$-line as $l_A$ to distinguish it. 
We will always assume $l_A$ has weight 1, and thus will denote our weight set 
$\beta \in \mathcal D(3,n+1)$ as $(\vec w, 1)$.  In this section, there is no need to restrict the set of weights $ \vec w$. However in the following sections, we will consider an additional restriction on the weights (see Definition \ref{def:adw}).

We have a naturally induced \emph{restriction} morphism  
$$ \varphi_A: M_{(\vec w,1)}(\bP^2, n+1)  \to  M_{0, \vec w},$$
induced by considering the intersection of $l_A$  with the lines
$l_i$ where $i \in \{ 1, \ldots, n \}$.  Next, we take the fiber of this restriction over a generic point $q \in M_{0, \vec w}$, and then take closure of this fiber in the compact moduli space of weighted hyperplane arrangements. 
\begin{definition}\label{def:Rn}
Let 
$q \in  {M}_{0,\vec w} \subset \overline M_{0, \vec w}$
 be a generic point. We define $\R_{\vec w}(q)$ as the closure in $\overline{M}_{\vec w, 1}(\bP^2, n+1)$ of the fiber product of the following diagram:
$$
\xymatrix{
R_{\vec w}(q) \ar[r] \ar[d] & \overline M_{(\vec w,1)}(\mathbb{P}^2,n+1)  \ar[d]^{\varphi_A}
\\
q  \ar[r] & \overline M_{0,\vec w}
}
$$
\end{definition}
\begin{remark}\label{rmk:BH}
B. Hassett gave an example of families 
$\left( \mathcal X,\frac{1}{2} \mathcal D \right) \to  
\operatorname{Spec} \left( \mathbb C [[ t]] \right)$ where  $\mathcal D |_{t=0}$ has embedded points. In general for pairs, the  components of the boundary with fractional coefficients $\leq \frac{1}{2}$ need not be Cohen-Macaulay.  By 
\cite[Lemma 1.5.1]{alexeev2013moduli}, the mentioned difficulty will \emph{not} occur for very generic coefficients of the form $\vec w$ for which one entry satisfies $w_i=1$.
\end{remark}

\begin{lemma}\label{dim}
The dimension $\dim(\R_{\vec w}(q) ) =  n-3$.
\end{lemma}
\begin{proof} 
By the fiber product  construction we see that 
$$
 \dim \left( \R_{\vec w}(q) \right)= 
\dim\left(  \overline{M}_{(\vec w,1)}(\PP^2,n+1) \right)
-
\dim( \overline{M}_{0,n} )
$$
The result  follows since  $ \dim \left( \overline{M}_{w}(\PP^2,n+1) \right) = 2(n-3)$
(see \cite[pg 84]{alexeev2013moduli}).  
\end{proof}

\begin{remark}\label{rmk:singularities} 
We will show in Proposition \ref{sing} that
\begin{enumerate}
\item the only singularities 
in the shas parametrized by $R_{\vec w}(q)$
are multiple points (no overlapping lines), as each line $l_i$ with $1 \leq i \leq n$ intersects the fixed line $l_A$ in a \emph{distinct} point.
\item The dual graph of $X$ is a rooted tree (see Proposition \ref{sing} [II]). This allows us to fully describe the shas parametrized by $\R_{1^n}(q)$ (see Figure \ref{fig:fig1}). 
\item Each \emph{broken line}  $l_i$ can be seen as a chain of lines that starts in the rooted component. The $l_i$ may have several branches, and can be contained in several components.
\end{enumerate}
\end{remark}

\begin{definition}\label{def:stabilize}
We say that the weight $\vec \beta$ \textbf{destabilizes} the multiple point $p(K)$  if the sum
$\displaystyle\sum_{k\in K} \beta_i > 2$.
We also say $\vec \beta$ destabilizes the sha $(X,D)$ if the pair has a singularity 
destabilized by $\vec \beta$.
\end{definition}

In what follows, we discuss the stable replacement of shas with multiple points which will be relevant for us
(see \cite[Chapter 5]{alexeev2013moduli} for a complete discussion). 
\subsection{Stable replacement}\label{stablerp}
Let $I \subset \{1, 2, ..., n\}$ be an index set. We
consider two chambers in $\mathcal D(3,n+1)$ separated by the wall
$W(I)$ as defined in Theorem \ref{thm:wallchamber}.  
Let $\vec w \leq \vec v$ be weights in those chambers 
such that $\sum_{i \in I} w_i <2$ and $\sum_{i \in I} v_i > 2$. 
Let $\vec u$ be a weight in the wall that separates those chambers, so in particular $\sum_{i \in I} u_i=2$.

Let  $(X,D)$ be a sha 
parametrized by $\overline M_{(\vec w, 1)}(\PP^2,n+1)$, and suppose that the sha 
has only a multiple point $p(I)$; notice that the point $p(I)$  will never be supported on $l_A$ (Remark \ref{rmk:singularities} (1)).
By (3) in Theorem \ref{thm:wallchamber}, changing the weights from
$\vec w$ to $\vec u$ will not modify the moduli spaces, so 
$$
\overline M_{(\vec w, 1)}(\PP^2,n+1) \cong \overline M_{(\vec u, 1)}(\PP^2,n+1).
$$
The singularity $p(I)$ is still  log canonical with respect to the weights
$(\vec u,1)$. Therefore, $(X,D)$ is in the universal family associated to weights
$\vec u$.

Next, we change the weights from $\vec u$ to $\vec v$.  By (2) in Theorem 
\ref{thm:wallchamber}, there is a contraction 
$$
\pi_{\vec v,\vec u}: 
\overline M_{(\vec v, 1)}(\PP^2,n+1)
\to
\overline M_{(\vec u, 1)}(\PP^2,n+1)
$$
By moduli theory, we know that the center of this morphism is the locus parametrizing shas with singularities that are destabilized respect to the new weights $(\vec v,1)$.  In particular, 
the  sha $(X,D)$ is no longer 
parametrized by $\overline M_{(\vec v, 1)}(\PP^2,n+1)$ because
$\sum_{i \in I} v_i > 2$.

Let $z \in \overline M_{(\vec u, 1)}(\PP^2,n+1)$ be the  point parametrizing the sha $(X,D)$. 
Next, we describe the sha $(\tilde{X}, \tilde{D})$ parametrized by a generic point in $\pi^{-1}_{\vec v,\vec u}(z)$.   We first blow up $X$ at $p(I)$, and we attach a $\bP^2$ along the exceptional divisor $E_{p(I)}$ to obtain a new  surface
$$
\tilde{X} = Bl_{p(I)}X \cup_{E_{p(I)}} \mathbb{P}^2
$$
with the lines $(l_i$ , $i \in I)$ crossing into the new $ \mathbb{P}^2$
and  defining a new divisor $\tilde{D}$
(see Figure \ref{examples2}).  The multiple lines defining $p(I)$ are  separated in $Bl_{p(I)}X$, and they are generically separated in the new component $\mathbb{P}^2$.  They may acquire a multiple point, but they cannot overlap with each other, because they are already separated in the double locus.

\begin{figure}[!htb]
\minipage{0.25\textwidth}
  \includegraphics[width=\linewidth]{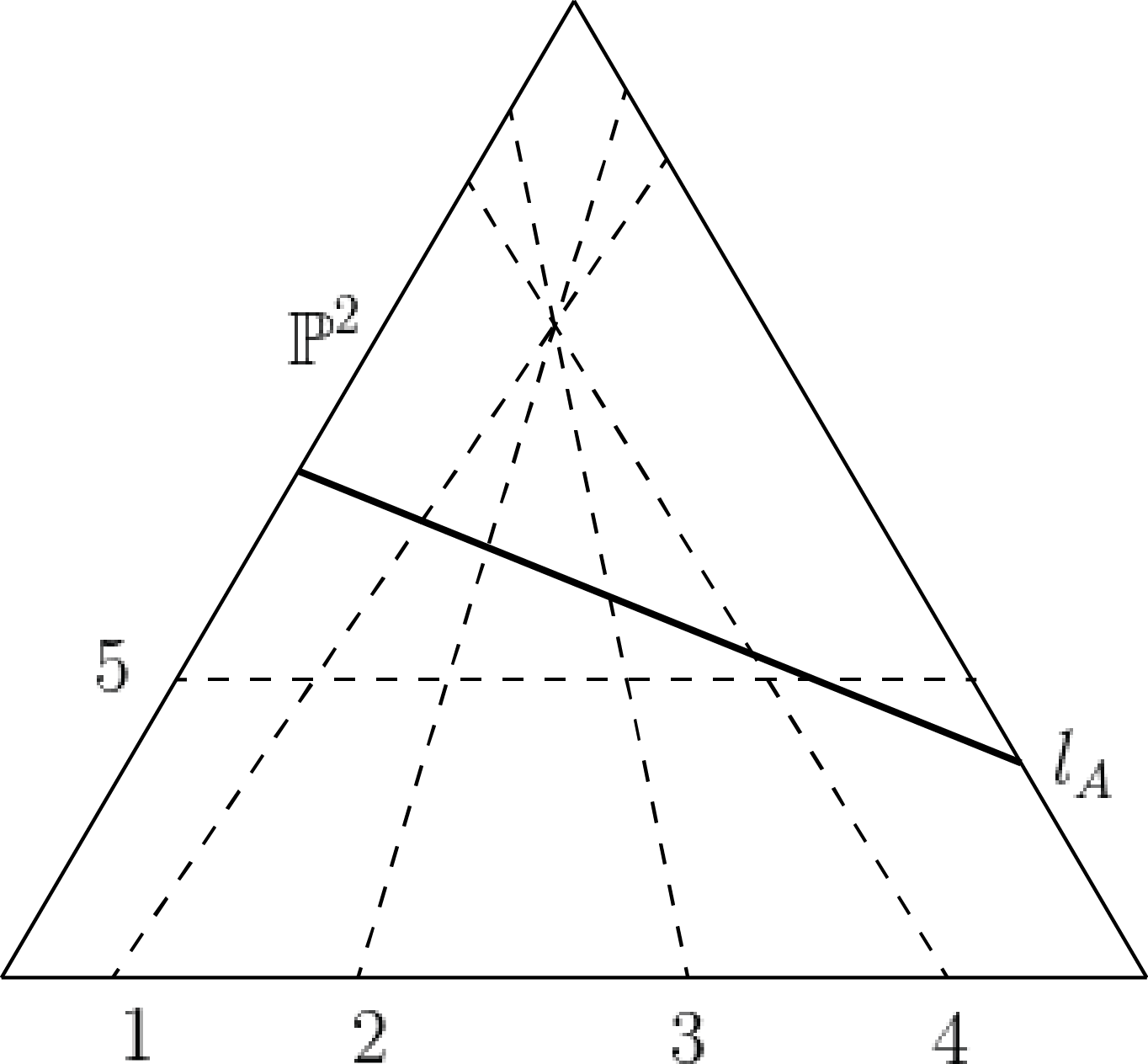}
\endminipage\hfill
\minipage{0.25\textwidth}
  \includegraphics[width=\linewidth]{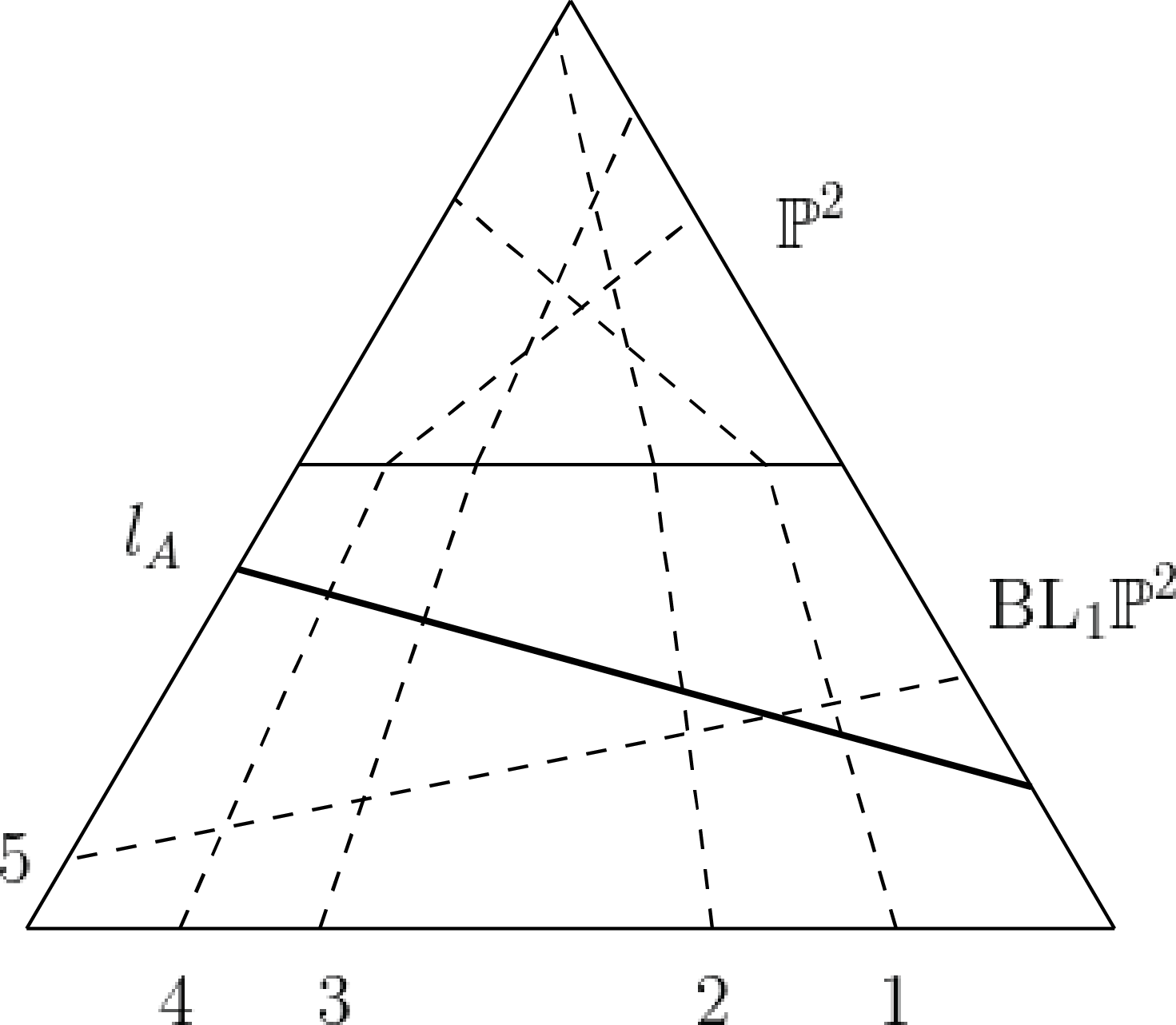}
\endminipage\hfill
\minipage{0.25\textwidth}%
  \includegraphics[width=\linewidth]{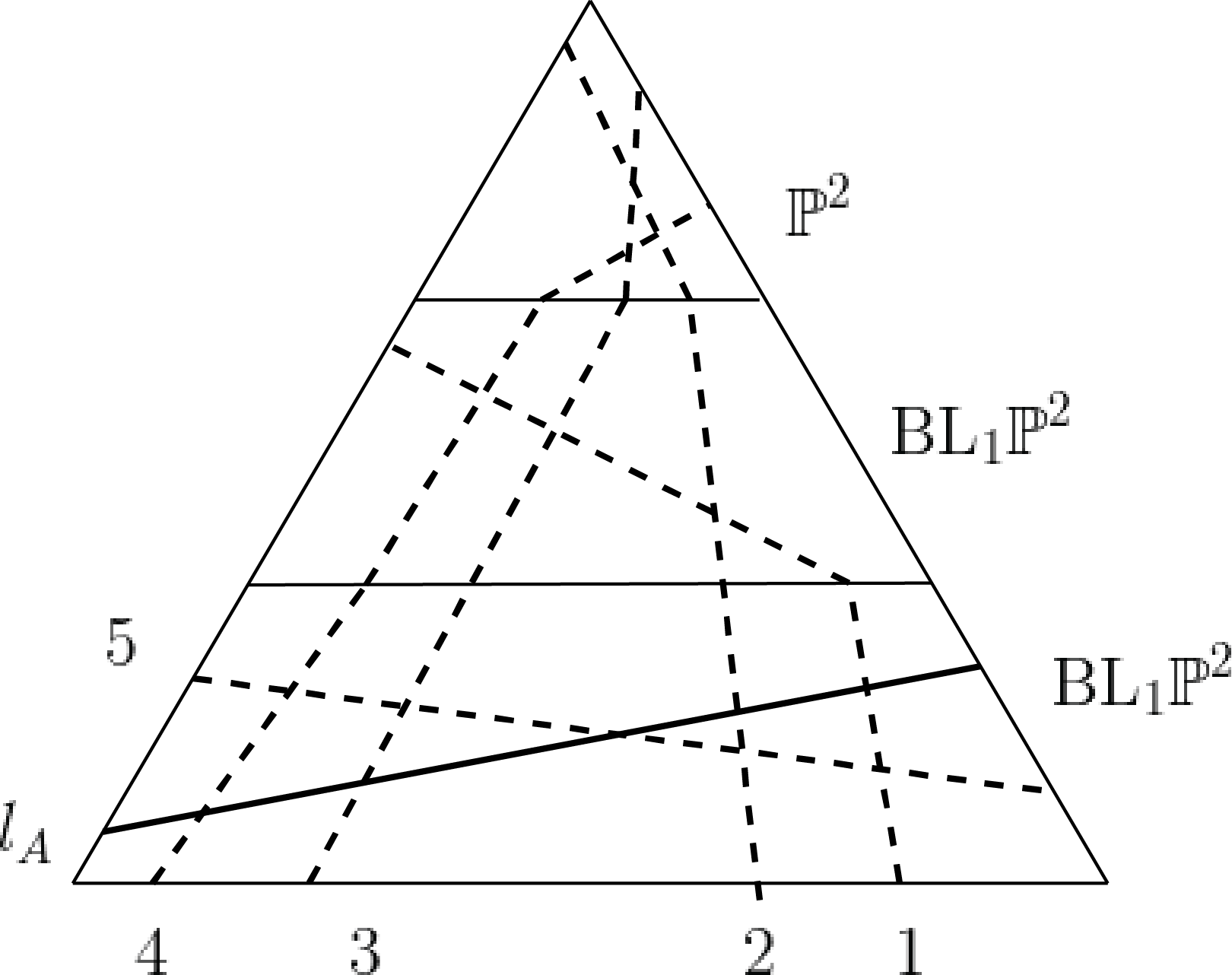}
\endminipage
\caption{
Quadruple point and its generic and non-generic stable replacement.
}
\label{examples2}
\end{figure}

\begin{example}\label{ex:stbr}
Consider a quadruple point in an arrangement of 6 lines-- then there are two possible stable replacements. The starting configuration is stable if the total weight of the intersection point of the four lines $l_1, .., l_4$ is $\leq$ 2. Increasing the weights of all the lines to one causes any singularity with multiplicity larger than two to  become unstable. Generically, the stable replacement has a new component where the 4 lines are separated. The four lines plus the double locus in $\bP^2$ have two dimensional moduli, so that we can further degenerate the configuration to a triple point. In this case, we must blow up the new component, obtaining a surface with three components. Here, the additional surface is a $\bP^2$ with three lines.  Since a configuration of three lines and the double locus in $\bP^2$ has no moduli, we cannot degenerate the configuration any further.  These two cases are all of the possible stable replacements.
\end{example}

 \section{$\R_{\vec w_0}$ as a GIT quotient and some properties of $\R_{\vec w}$}\label{construct}
 The starting point of this section is Lemma \ref{base}, where we show that there are weights $\vec {w_0}$ such that  $\R_{\vec {w_0}}(q) \cong \bP^{n-3}$.  Afterwards, we study some geometric properties of $\R_{\vec w}$ in general, such as the surfaces parametrized and the singularities that appear (Proposition \ref{sing}), as well as the outcome of wall-crossing on our moduli spaces (Lemma \ref{fibers}). 
 
The results of this section do not depend on the $q$ used in the definition of 
$\R_{\vec w}(q)$, so we simplify our notation and we just write  $\R_{\vec w}$.  First, we define our admissible weights.

\begin{definition}\label{def:adw}
Let  $\vec w_{0}=(w_{0_1}, \ldots, w_{0_n})$ be a set of rational numbers such that for every subset $I \subsetneq \{1, \ldots, n\}$ the inequality  $\sum_{i \in I} w_{0_i} \leq 2 $ holds. 
The \textbf{set of admissible weights} is
$$
\mathcal D^{R}_n= 
\{
(w_1, \ldots, w_n) \in \mathbb Q^n \; | \; 1 \geq w_i > 0, \; \sum_{i=1}^n w_i \geq 2, \;  
w_i \geq w_{0_i}
\}
$$
\end{definition}

The chamber decomposition of $\mathcal D(3,n+1)$ induces a chamber decomposition on $\mathcal D^R_n$ where the chambers are separated by the walls $W(I)$ (see Theorem \ref{thm:wallchamber}).

\begin{definition}\label{def:adj}We say that two weights $\vec v$ and $\vec u$ are \textbf{adjacent} if each of them belongs to a chamber in $\mathcal D^R_n$ and those chambers are separated by a single wall $W(I)$.  Sometimes, we say that the weights $\vec u$ and $\vec v$ are \textbf{separated} by $W(I)
$. \end{definition}

Moreover, by Remark \ref{rmk:BH}, to avoid any subtle technicalities, we will assume all our weighs are very generic.\\

Before showing that $\R_{\vec w_0} \cong \mathbb{P}^{n-3}$ (Lemma \ref{base}), we prove a key lemma.

\begin{lemma}\label{lemma1} The subgroup of $\operatorname{SL}(3, \mathbb C)$ that fixes:
\begin{itemize}
\item  three lines $l_n$, $l_{n-1}$ and $l_A$ in general position, and 
\item  $n$ distinct points $\{ l_1 \cap l_A, \ldots, l_n \cap l_A \}$ in $l_A$.
\end{itemize}
is equal to $\mathbb{C}^*$.\end{lemma}

\begin{proof}
We can suppose without lost of generality that the lines are 
\begin{align*}
l_A:=(x_0=0), & &  l_{n-1}:=(x_1=0), & & l_n:=(x_2=0)
\end{align*}
The subgroup that fixes those lines in $\bP^2$ is $(\bC^*)^2$, and it is given by  matrices of the form  $g= \operatorname{diag}((g_2g_1)^{-1},g_1, g_2)$ which acts on any point in the line $l_A$ by $g \cdot [0:q_1:q_2] \to [0:g_1q_1:g_2q_2].$ By hypothesis, the points $\{ l_1 \cap l_A, \ldots, l_n \cap l_A \}$ on $l_A$ are fixed, 
implying that $g_1=g_2$. \end{proof}

\begin{lemma}\label{base}
Let $\vec w_{0}$ be as in Definition \ref{def:adw}. Then  
$$
\R_{\vec w_{0}} 
\cong
 \mathbb{P}^{n-3} 
  \subset \overline{M}_{(\vec w_{0},1)}(\mathbb{P}^2,n+1),
$$ 
and each fiber of the universal family over $\R_{\vec w_{0}}$
 is a pair
$\left(  \bP^2, \sum_{k=1}^n  w_{0_k}l_k +l_A \right)$
such that 
\begin{enumerate}
\item the $n$ lines $l_i$ cannot all meet at an $n$-tuple point, 
\item any multiple point of multiplicity strictly smaller than $n$ is allowed,
\item none of the lines $l_i$ can overlap with $l_A$, 
\end{enumerate}
\end{lemma}
\begin{proof}
Let $l_A$ be the line with weight $w_A =1$ that induces the restriction morphism 
$$
M_{(\vec w_0,1)}(\mathbb{P}^2,n+1)
\to 
M_{0,\vec w_0}.
$$

To prove (1), recall that an $n$-tuple point is unstable if and only if the sum of the weights $\sum_{i=1}^n w_{0_i} > 2$, which is true by assumption. 

Following the proof of (1), we note that (2) holds because of the assumption that for every subset $I \subsetneq \{1, ..., n\}$ the sum of the weights is $\leq 2$. 

To prove (3), we recall that a multiple line is unstable if the sum of the weights is greater than 1. Since the weight $w_A$ of the line $l_A$ is already 1, no other line can overlap with it.

Let $(X,D)$ be any  configuration parametrized by $\R_{\vec w_0}$.  By (1) and (3), we can suppose that the lines $l_{n-1}$, $l_n$ and $l_A$ are fixed and in general position.  By definition, the points
$\{ l_1 \cap l_A, \ldots, l_n \cap l_A \} \subset l_A$
induce  the equivalence class $q \in M_{0,n}$, and thus we can  fix these points.  

We can now demonstrate that $\R_{\vec w_0} \cong \bP^{n-3}$.  First note that the parameter space of each line $l_i$ with $1 \leq i \leq n-2$ is $\bA^1$, because the intersection $l_i \cap l_A$ is fixed.
We can choose coordinates on each $\bA^1$ so that the point $0 \in \bA^1$  parametrizes whenever the line $l_i$ coincides with the fixed intersection $l_n \cap l_{n-1}$.   Then the parameter space of the $(n-2)$ lines $l_1, \ldots, l_{n-2}$ is 
$(\bA^1)^{n-2} \setminus (0, \cdots  , 0)$, since we cannot have an $n$-tuple point by (1). Therefore, by Lemma \ref{lemma1}, we conclude that $$
\R_{\vec w_0} \cong \bA^{n-2} \setminus (0, \cdots, 0) \sslash \bC^* \cong \bP^{n-3}.
$$\end{proof}

Next,  we construct a family of shas over $R_{\vec w_0}$.  Before doing that, we set up some notation.

\begin{notation}\label{notation}
We choose a coordinate system $[t_0: t_1 : t_2] \in \bP^2$ such that:
\begin{align*}
l_A:=(t_0=0),
& &  
l_{n-2} \cap l_A:=[0:0:1], & & l_{n-1}:=(t_2=0), 
&& 
l_{n}:=(t_1-t_2=0).
\end{align*}
and we select the point  $q \in M_{0, \vec w_0}$ induced by the 
following configuration of points in $l_A$
\begin{align*}
\{  [0:a_1:1], \ldots,  [0:a_{n-3}:1], [0:0:1], [0:1:0], [0:1:1]   \},
\end{align*}
Under this choice of coordinates,  $[s_1: \ldots : s_{n-2}] \in R_{\vec w_0}(q)$ parametrizes the following configuration of lines
with $1 \leq i \leq (n-3)$
\begin{align*}
l_i:=(t_1-a_it_2+s_it_0=0), & & l_{n-2}:=(s_{n-2}t_0+t_1=0),
&& l_{n-1}:=(t_2=0), && l_n:=(t_1-t_2=0).
& &
\end{align*}
\end{notation}

In the following lemma, we consider $R_{\vec w_0} \cong \mathbb P^{n-3}$ with coordinates $[s_1, \ldots, s_{n-2}]$ as above and the projective space $\mathbb P^{n-1}$ with coordinates 
$[z_1, \ldots, z_{n}]$.
We exclude the $n=4$ case for convenience of notation (see Remark \ref{rmk:basecases}).

\begin{lemma}\label{lemma:universal} 
For $n \geq 5$, let $\mathcal U_{\vec w_0}$ be the blow up of $\mathbb P^{n-1}$ at the line defined by
$$
Z:=\{ z_{k}-z_{k+2}  =0 \; | \; 1 \leq k \leq n-2 \}.
$$
and let  $\sigma_i$ be the strict transform of the following $n$ hyperplanes in $\mathbb P^{n-1}$ with $1 \leq i \leq n-3$.
\begin{align*}
H_i &:=(a_2z_3-a_1z_4)-a_i(z_3-z_4)+(a_2-a_1)(z_i-z_{i+2}) = 0
\\
H_{n-2} &:= (a_2-a_1)(z_{n-2}-z_n)+a_2z_3-a_1z_4 = 0 
\\
H_{n-1} &:=  z_3-z_4 =0 
\\
H_n &:= (a_2-1)z_3-(a_1-1)z_4 =0 
\end{align*}
Then there exists a flat, proper morphism
$
\phi_{\vec w_0}\; :
\mathcal U_{\vec w_0} \to R_{\vec w_0}
$
such that for every $\vec s \in  R_{\vec w_0}$ the fiber
$\phi_{\vec w_0}^{-1}(\vec s)$  is isomorphic to $\mathbb P^2$.
Moreover, if $E_{\vec w_0} \subset \mathcal U_{\vec w_0}$   is the exceptional divisor, then the configuration of lines
\begin{align*}
l_i:= \phi_{\vec w_0}^{-1}(\vec s) \cap  \hat\sigma_i
& &
l_A:=\phi_{\vec w_0}^{-1}(\vec s) \cap E_{\vec w_0}
\end{align*}
define the stable sha of weight $\vec w_0$ parametrized by $\vec s$.
\end{lemma}
\begin{proof}
Let 
$\pi_Z: \mathbb P^{n-1} \to R_{\vec w_0}$ be the projection defined by $\{s_k= z_{k}-z_{k+2} \; | \; 1 \leq k \leq n-2\}$.  Note that $Z$ is the indeterminacy loci of $\pi_Z$, and that given a point $\vec s \in R_{\vec w_0}$, we have  $\pi_Z^{-1}(\vec s) \cong \mathbb P^2$. 
Therefore, the map $ \mathcal U_{\vec w_0} \to R_{\vec w_0}$ is a  $\mathbb P^2$-fibration  obtained by the composition 
$\mathcal U_{\vec w_0} \to \bP^{n-1} \to R_{\vec w_0}$.  

The following functions with $2 \leq m \leq \frac{n}{2}$  if $n$ is even, 
and $2 \leq m \leq \frac{(n+1)}{2}$ if $n$ is odd.
\begin{align*}
\zeta_1&=t_1-a_1t_2+s_1t_0,
&
\zeta_2&=t_1-a_2t_2+s_2t_0,
\\
\zeta_{2m-1}&= \zeta_1 - t_0 \sum_{k=0}^{m-2} s_{2k+1},
& 
\zeta_{2m} &= \zeta_2 - t_0 \sum_{k=1}^{m-1} s_{2k}
\end{align*} 
define, for a fixed $\pi^{-1}_Z(\vec s)$, a map  
$
\zeta_{\vec s}:
\mathbb P^2 \to \pi^{-1}_Z(\vec s)
$ 
given by  
$$
\zeta_{\vec s}:
 [t_0,t_1,t_2] \to [\zeta_1, \zeta_2, \ldots,\zeta_n ].
$$
Indeed, we can verify the image of the map $\zeta_{\vec s}$ is $\pi_Z^{-1}(\vec s)$ since  
\begin{align*}
\pi_Z\left( \zeta_{\vec s}[t_0,t_1,t_2] \right)
&=
[
\zeta_1-\zeta_3, \zeta_2-\zeta_4, \ldots, \zeta_{n-2}-\zeta_n]
=
[
s_1t_0, s_2t_0, \ldots, s_nt_0
].
\end{align*}
We also note that the map is not defined for $(t_0=0)$ because 
$\zeta_{\vec s}^{-1}(Z)=(t_0=0)$.
Moreover, by the definition of the $H_i$ above,  
and the equations of the lines given in Notation \ref{notation}
it holds that
\begin{align*}
\zeta_{\vec s}(l_i)= \pi_Z^{-1}(\vec X) \cap H_i
& &
\zeta_{\vec s}(l_A)  &= Z
\end{align*}
These equalities follow at once by observing that
$\zeta_3=\zeta_1-t_0s_1$, $\zeta_4=\zeta_2-t_0s_2$ as well as
\begin{align*}
a_2\zeta_3-a_1\zeta_4
&=
(a_2-a_1)t_1
& 
\zeta_3-\zeta_4
=
(a_2-a_1)t_2
& &
\zeta_i-\zeta_{i+2} =s_{i}t_0.
\end{align*}
Finally, we assign the weights given by $\vec w_0$ to the $n$ hyperplanes and weight 1 to the exceptional divisor, we get a family of shas with respect to the weights $\vec w_0$.
\end{proof}

\subsection{Generalities on $\R_{\vec w}$} We start with a explicit description of the surfaces parametrized by 
$\R_{\vec w}$.

\begin{proposition}\label{sing}
Let $(X,D)$ be a sha parametrized  by $\R_{\vec w}$, then the following hold:
\begin{itemize}
\item[I] The only  singularities  in $(X,D)$ are of the form 
$p(J)$ (see Notation \ref{notate}). 
In particular,  the shas never have overlapping lines. 
\item[II] The dual graph $Graph(X)$ of $X$ is  a rooted tree where the rooted vertex is the unique
surface containing the  line $l_A$. 
\item[III] All the components of  $X$ are a blow up of  $\mathbb P^2$ at $k \geq 0$ points. In particular, the stable replacement of any sha 
parametrized by $R_{\vec w}$ is obtained by blowing up isolated points. That is, we never have to blow down a $(-1)$-curve. 
\end{itemize}
\end{proposition}
\begin{proof}
Let $\vec w \in \mathcal D^R_n$ be an admissible weight and consider a 
sequence of weights $\vec \gamma_1, \ldots,\vec \gamma_m$ such that 
$\vec \gamma_1:=\vec w_0$, $\vec \gamma_m :=\vec w$, 
the weights $\vec \gamma_{i} \leq \vec  \gamma_{i+1}$ 
are adjacent to each other (see Definition \ref{def:adj}),
and $m$ is the minimal length of such sequences. We prove  our proposition by induction on  $m$. The case $m=1$ follows from
Lemma \ref{base}. In that case, the dual graph  for every pair is a point. 

We suppose the statement holds for $m-1$. Let 
$\vec \gamma_m:=\vec w$
and let $\vec \gamma_{m-1}:= \vec v$ be two adjacent weights separated by the wall $W(I)$.  
We highlight that walls of type $\widetilde W(K)$  
in $\mathcal D(3,n+1)$ do not modify neither  $\R_{\vec v}$  nor the shas parametrized by it because the space $\R_{\vec v}$ only parametrizes shas with isolated multiple points by our inductive hypothesis.  
By case (2) in Theorem \ref{thm:wallchamber}, there is a contraction 
$$
\pi_m: \overline{M}_{(\vec w,1)}(\mathbb{P}^2,n+1)  \to 
\overline{M}_{(\vec v,1)}(\mathbb{P}^2,n+1).
$$
Let $(X',D')$ be an arbitrary sha with at least one $p(I)$ singularity
and parametrized by a point $z \in R_{\vec v}$.   We will show that 
any shas   $(X',D')$ parametrized by $\pi^{-1}_m(z)$
have only multiple point singularities.  

By  Subsection \ref{stablerp}, the fibers of  $\pi_{m}$ parametrize a new sha $(X,D)$ containing a new $\mathbb{P}^2$ component with  the lines 
$\{ l_{i_1} \ | \ i_k \in I \}$. 
 Therefore, the fiber of $\pi_{m}$ over the point parametrizing $(X',D')$ is the moduli associated to the pairs
$
\left( \mathbb{P}^2, l_{i_1}+\ldots + l_{i_k} \right)
$
that satisfy the following conditions:
\begin{enumerate} 
\item The lines cannot all overlap in an $|I|$-tuple point, because this is precisely the singularity we destabilized.
\item The pair can have any singularity of the form $p(J):=\cap_{i_k \in J} l_{i_k}$ with $J$ properly contained in $I$, because we are only destabilizing one type of singularitiy. We must cross more walls to destabilize $p(J)$. 
\item  
Let $H_{0}$ be the hyperplane obtained by intersecting the new  $\mathbb{P}^2$ with the other components of $\tilde X$. 
Then the  lines $l_{i_s}$ cannot overlap with $H_0$.
\item The equivalence class induced by the intersection of the lines $l_{i_s}$ with the gluing locus is fixed because 
the sha $(X', D')$ is fixed. 
\end{enumerate}
These are precisely  the same conditions used in the proof of Lemma \ref{base} with the gluing locus playing the role of $l_A$. Therefore,  every
positive dimensional  fiber of $\pi_m$ is isomorphic to $\mathbb{P}^{(|I|-3)}$.  The new shas $(X,D)$  have at worst mulitple point singularities, because the lines 
$\{ l_{i_1} \ | \ i_k \in I \}$ cannot overlap in the new component $\mathbb{P}^2 \subset X$ by 
the fourth condition above.  The singularities of $(X,D)$ away from this $\bP^2$ are also multiple points by our hypothesis on the singularities of $(X',D')$.\\

Part (II) follows from the previous argument because
the wall crossing between two adjacent  weights $\vec v$ and $\vec u$ adds a new vertex  to $Graph(X')$ corresponding to the new $\bP^2$.   The multiple points never occur in $l_A$, and so $l_A$ is always contained in a single surface which will be our root.\\  

Finally, we prove Part (III). In the absence of overlapping lines, as in our case,  \cite[Thm 5.7.2 (ii)]{alexeev2013moduli} states that
a $\mathbb P^1 \times \mathbb P^1$ component
is only obtained from a configuration of points 
 with the following characteristics:
\begin{enumerate} 
\item Given a $\mathbb P^2$-component with lines 
$\{l_i \}$, 
there are exactly two non-log-canonical points in the configuration of those lines.
\item The line $l_k$ between the two-non log canonical points have weight 1.
\item There is not an additional line $l_s$  or a component of the double locus intersecting $l_k$ transversally. 
\end{enumerate}
Under the above conditions, one must blow up the two points and contract the strict transform of the line between them (see \cite[Figure 5.8]{alexeev2013moduli}).

To clarify this last condition, the reader should compare the following shas from  \cite[Fig 5.12]{alexeev2013moduli}.  In sha \#3, line $l_3$ intersects $l_4$ and  prevents a line from being contracted in the $\textrm{Bl}_2 \bP^2$ component, so that we do \emph{not} obtain a $\bP^1 \times \bP^1$. In contrast,  in sha \#8, there does not exist a similar line intersecting $l_1$, in which case the sha has a $\bP^1 \times \bP^1$ as the corresponding component.  \\

In particular, condition (3) will never happen in our case, as we always have either the double locus or the line $l_A$ intersecting the line $l_k$ transversally.  \end{proof}

The following result will be important for proving that 
$R_{\vec w}$ is smooth.
\begin{lemma}\label{fibers}
Let $\vec v \geq \vec u$  be adjacent weights in $\mathcal D^R_n$
separated by the wall $W(I)$. 
Let 
$$
\pi_{\vec v, \vec u}: \overline{M}_{(\vec v,1)}(\mathbb{P}^2,n+1)  \to 
\overline{M}_{(\vec u,1)}(\mathbb{P}^2,n+1)
$$
 be the associated wall crossing morphism.  Then its restriction 
$\phi_{\vec v, \vec u}: \R_{\vec v}  \to \R_{\vec u} $  has (scheme-theoretic) fibers equal to  $\mathbb{P}^{(|I|-3)}$.
\end{lemma}

The morphism $\phi_{\vec v, \vec u}$ has positive dimensional fibers over the loci parametrizing shas that become unstable with respect to the weights $\vec v$.  In our case, those are the shas with a 
isolated multiple point $p(I)$ and its fibers are described in  the proof of Proposition \ref{sing}. We now prove this scheme-theoretically.

\begin{proof}[Proof of Lemma \ref{fibers}]
Let $\phi_{\vec v, \vec u}: R_{\vec v} \to R_{\vec u}$ be the wall crossing morphism where $\vec v \geq \vec u$, let $A$ be the spectrum of an Artinian ring, and let $\psi: A \to R_{\vec v}$ be a deformation of $R_{\vec v}$. Furthermore, suppose that the total space of the composition $\phi_{\vec v, \vec u} \circ \psi: A \to R_{\vec u}$ is constant. We wish to show, by contradiction, 
that this forces the total space of $\psi: A \to R_{\vec v}$ to be the trivial deformation as well.\\

We may assume that the total space of $\phi_{\vec v, \vec u} \circ \psi: A \to R_{\vec u}$  is the trivial deformation of a pair $(X,D)$ where $(X,D)$ is stable with respect to the weights $\vec u$ but unstable with respect to $\vec v$. Indeed, if $(X,D)$ was stable with respect to both weights, then the morphism $\phi_{\vec v, \vec u}: R_{\vec v} \to R_{\vec u}$ is an isomorphism on this locus, and there is nothing to prove. \\

In particular, there exists $D' \subset D$ such that $D' = \cup_{i \in I} L_i$ with $\sum_{i \in I} u_i \leq 2$ and $\sum_{i \in I} v_i > 2$. Then the definition of $\phi_{\vec v, \vec u}: R_{\vec v} \to R_{\vec u}$ implies that the preimage of the sha $(X,D)$ is $(Y, D_Y + Z)$, where $Y = X' \cup \bP^2$ with $X' = \mathrm{Bl}_{p(I)}X$. Recall that $p(I)$ denotes the point we are required to blowup, as there are too many weighted lines passing through that point with respect to $\vec v$. \\

If we denote the gluing locus
 by $Z_1 \subset X'$ and $Z_2 \subset \bP^2$, then  it suffices to show that the deformation restricted to the three pairs, $(X', Z_1), (\bP^2,Z_2)$, and $Z = Z_1 \cong Z_2$ (the gluing locus $X' \cap \bP^2$) is trivial. Indeed, we first note that $(\bP^2, Z_2)$ is rigid. Furthermore, the deformation restricted to $(X', Z_1)$ is trivial, as the pair $(X', Z_1)$ is uniquely determined by $(X,D)$, which is assumed to be fixed. In particular, $(X', Z_1)$ is obtained as the blowup of a fixed variety at a fixed point. Therefore, it suffices to show that the deformation is trivial on the gluing locus, $Z$. To do so, we recall how our construction yields this line $Z$.\\

Recall that we are blowing up a point $p(I)$ inside a surface $X$ living inside a total space $\bar{X} := X \times A$. In particular, there is an inclusion of normal bundles $$N_X := N_{p(I)/ X} \subset N_{A/ \bar{X}} := N_{\bar{X}},$$ where $N_{X}$ is also the restriction of $N_{\bar{X}}$ on $X$. Indeed, we obtain $N_{A/\bar{X}}$ as we are blowing up a $p(I)$ inside each fiber, and an entire family of them, thus blowing up a section $A \subset \bar{X}$. The exceptional divisor of the blowup of $p(I)$ inside $X \subset \bar{X}$, is defined by the projectivization of these normal bundles -- indeed, the $\bP^2$ arises from the projectivization of $N_{\bar{X}}$, and the gluing locus $Z \cong \bP^1$ arises from the projectivization of $N_X$. \\

As $\phi_{\vec v, \vec u}  \circ \psi$ is assumed to be the trivial deformation, the normal bundles $N_X$ and $N_{\bar{X}}$, as well as the inclusion
$N_X \to N_{\bar X}$
 never change.  Now it suffices to note that any non-trivial deformation of $Z$, when composed with the wall crossing $\phi_{\vec v, \vec u}$,  would change the inclusion
$N_X \to N_{\bar X}$, thus contradicting the fact that $\phi_{\vec v, \vec u} \circ \psi$ is a trivial deformation.\\

Therefore, the moduli is determined by the moduli of the lines $\sum_{i \in I} L_i + Z$ inside $\bP^2$, such that $\sum_{i \in I} L_i = 2 + \epsilon$ and $L_{I} \cap Z$ is a fixed point of $M_{0,n}$, which is $\bP^{|I| - 3}$ by Lemma \ref{base}. 
\end{proof}

\section{Construction of $\R_{\vec w}$ via wonderful compactifications}\label{sec:wonderful}

As in the previous section, the results of this section do not depend on the $q$ used in the definition of 
$\R_{\vec w}(q)$,as long as it is a generic point of $M_{0, \vec w_0}$. We simplify our notation and just write  $\R_{\vec w}$.\\

Recall in Notation \ref{notation} we showed that the equivalence class of the $n$ lines parametrized by  $[s_1: \ldots : s_{n-2}] \in \R_{ \vec w_0}$  is induced by the lines
\begin{align*}
l_i:=(x_1-a_ix_2+s_ix_0=0), & & l_{n-2}:=(s_{n-2}x_0+x_1=0),
&& l_{n-1}:=(x_2=0),
\\
l_A:=(x_0=0),
 && l_n:=(x_1-x_2=0).
& &
\end{align*}
Therefore, the point $[1:0: \ldots: 0] \in  \R_{\vec w_0}$ parametrizes a pair with an $(n-1)$-tuple point at $[1:0:0] \in \mathbb{P}^2$ induced by the intersection of the lines 
$
l_2, \ldots, l_n.
$
Similarly, the hyperplane $(s_1=0) \subset \R_{\vec w_0}$ parametrizes a pair with a triple point at $[1:0:0]$. \\

We now show that this behavior holds in general.

\begin{lemma} \label{bij}
For every $I \subset \{1, \ldots, n\}$,  there  is a linear subspace  $\mathbb{P}^{(n-|I|-1)} \cong H(I) \subset \R_{\vec w_0}$ 
that generically parametrizes  a configuration with an $|I|$-tuple point $p(I)$ given by the intersection of the lines $\{ l_{i}  \; | \; i \in I \}$.  
\end{lemma}
\begin{proof} A set of lines  $\{ l_i \; | \; i \in I \}$ has an $|I|$-multiple point if and only their dual points $\{ y_{i} \; | \;  i \in I \}$ are collinear. Taking any subset of three of these points, the associated matrix $[y_j, y_k, y_l]$ has determinant equal to zero.  In particular, these equations are linear on the variables $s_i$ and define $H(I)$. Finally, the dimension count is $(n-3)- (|I|-2)=n-|I|-1$.\end{proof}

\begin{example}We use the equation of the lines as given in Notation \ref{notation}. For example associated to the points $y_1=[s_1,1,-a_1]$, $y_2=[s_1,1,-a_2]$, and  $y_3=[s_1,1,-a_3]$, we have the equation
$$
s_1(a_2-a_3)-s_2(a_1-a_3)+s_3(a_1-a_2)=0.
$$
\end{example}

The sets $H(I)$ will generate the centers of the morphism
$\R_{\vec 1^n} \to \R_{\vec w_0}$. These morphisms are induced by changing the weights, and the description of these linear subspaces will be crucial for the next subsection.

\subsection{Wonderful compactifications}\label{sec:wonderful} In what follows, we review the pertinent definitions of \emph{wonderful compactifications} following \cite{li}. We note that the theory of wonderful compactifications originiated in \cite{de1995wonderful}.

\begin{definition}\label{clean}
An \textbf{arrangement} of subvarieties of a nonsingular variety $Y$  is a finite set $\mathcal{S}=\{S_i\}$ of nonsingular closed  subvarieties $S_i \subset Y$ closed under scheme-theoretic intersection.    Given  $\dim(Y)=(n-3)$,
we say that a finite collection of $k$ nonsingular subvarieties $S_1, ..., S_k$ intersect \textbf{transversely}, if  either $k=1$ or for any $y \in Y$ the following conditions holds
(see \cite[Sec 5.1.2]{li})
\begin{itemize}
\item[(a)] there exist a system of local parameters $x_1, ..., x_{(n-3)}$ on $Y$ at $y$ that are regular on an affine neighborhood $U$ of $y$ such that $y$ is defined by the maximal ideal $(x_1, ..., x_{(n-3)})$ as well as 
\item[(b)] integers $0 = r_0 \leq r_1 \leq ... \leq r_k \leq (n-3)$ such that the subvariety $S_i$ is defined by the ideal 
$$
(x_{r_{i-1}+1}, x_{r_{i-1}+2}, ..., x_{r_i})
$$ for all $1 \leq i \leq k$
\end{itemize}
If $r_{i-1} = r_i$ then the ideal is assumed to be the ideal containing units, which means geometrically that the restriction of $S_i$ to $U$ is empty.
\end{definition}
\begin{definition} \label{factors}
A subset $\mathcal G \subset \mathcal S$ is called a \textbf{building set} of $\mathcal S $ if for all 
$S_k \in\mathcal S$, the minimal elements of $\mathcal G$ containing $S_k$ intersect transversally and their intersection is equal to $S_k$ (by convention, the condition is satisfied if $S_k \in \mathcal G$). These minimal elements are called the \textbf{$\mathcal G$-factors} of $S_k$.  Let $\mathcal G$ be a building set and set $Y^o:=Y \setminus \bigcup _{S_k \in \mathcal G} S_k$. The  closure of the image of the natural locally closed embedding 
 (\cite[Def 1.1]{li})
\begin{center}
$Y^o \hookrightarrow  \prod \limits_ {S_k \in\mathcal G} Bl_{S_k} Y$
\end {center}
\vspace{0.1in}
is called the \textbf{wonderful compactification} of $Y$ with respect to $\mathcal G$. 
\end{definition}

\begin{theorem} \cite[Theorem 1.3]{li}\label{thmLi}
Let $\mathcal G$ be a building set and let $Bl_{\mathcal G} Y$ be the wonderfuld compactification of $Y$ with respect to $\mathcal G$. Then $Bl_{\mathcal G}Y$ is smooth  with normal crossing boundary,  and that for each $S_k \in \mathcal{G}$ there is a nonsingular divisor $D_{S_k }\subset Y_{\mathcal{G}}$. Moreover,  the union of the divisors is $Y_{\mathcal{G}}\setminus Y^o$, and any set of these divisors, with nonempty interesction, meet transversally.
\end{theorem}
\begin{example}\label{bR5}
A building set $\mathcal{H}$ in $R_{\vec w_0}$  is given by $5$ points $H(J)$ with $|J|=4$ and   10 lines $H(I)$ with $|I|=3$ parametrizing configurations with either a quadruple or a triple point respectively.  The arrangement $\mathcal{S}$ is the set of all possible intersections among them.  The 10 lines, which are not in general position,  intersect along $20$ points given by:
\begin{enumerate}
\item The point $H(I) \cap H(J)$ with $ |I  \cap J| =2 $ parametrizes the quadruple point $p(I \cup J)$.
\item The point $H(I) \cap H(J)$ with $ |I  \cap J| =1 $ parametrizes a configuration with two 
triple points  associated to $I$ and $J$.  There are 15 of these points.
\end{enumerate}
\end{example}
The above example illustrates the general behavior. 
\begin{lemma} \label{build}
Let $\mathcal S_{\vec w}$ be the set of all possible intersections of collections of subvarieties from
$$
\mathcal{H}_{\vec w }=\{ H(J)
\ | \
\sum_{i \in J} w_i > 2,   \ |J| \subset \{1, \ldots, n \}
\}.
$$
Then, $\mathcal S_{\vec w}$ is  an arrangement and $\mathcal H_{\vec w}$ is a building set.
\end{lemma}
\begin{proof}
 $\mathcal S_{\vec w}$ is an arragement by Definition \ref{clean}. 
For the last statement, let $S_k$ be an arbitrary element of $\mathcal S_{\vec w}$.  By definition,  $S_k$ is  an arbitrary nonempty intersection $S_k:=H(I_1)\cap \dots \cap H(I_m) $.  We need to prove two conditions: (I) that the minimal elements of $\mathcal{H}_{\vec w}$ that contain $S_k$ intersect transversally, and (II) that their intersection is equal to $S_k$.   

For (I), we first observe that  any $S_k$  can be written uniquely as an intersection of the form $H(J_1) \cap\dots \cap H(J_s)$, where  
$|J_i \cap J_k| \leq 1$ and each of the $J_i$ is a union of $I_j$.  Indeed, if $|I_1 \cap I_2|  \geq 2$
and $I_1 \cap I_2 \neq \{ 1, \ldots, n \}$, then their intersection must parametrize an $(|I_1|+|I_2|)$-tuple point. 
This implies that $H(I_1) \cap H(I_2) $ is either the empty set or $H(I_1 \cup I_2)\in \mathcal{H}_{\vec w}$.  In the latter case, we  can dismiss 
$H(I_1)$ and $H(I_2)$ while keeping $H(I_1) \cap H(I_2)$.  Iterating this process, we can find all the minimal elements $J_i \in \mathcal H_{\vec w}$ containing $S_k$.

Part (I) now reduces to showing that the intersection of the linear subspaces $\mathbb P^{(n-|J_i|-1)}$, $1 \leq i \leq s$, along $S_k$ is transversal.
By Definition \ref{clean}, it is enough to exhibit numbers 
$0=r_0 \leq r_1 \leq \ldots \leq r_s \leq (n-3)$ that satisfy the conditions of the aforementioned definition.
We can take 
\begin{align*}
&
r_0 :=0
,& 
r_m := \sum_{i=1}^m (|J_1|-2) 
& &
\text{ with } \qquad{} 1 \leq m \leq s.
&
\end{align*}
Indeed,  $r_s \leq (n-3)$ because 
$$
0
\leq
\dim(S_k)
=
(n-3)-  \sum_{i=1}^s (|J_1|-2) 
$$
since $S_k$ is non-empty.  We can take the linear subspace
$
H(J_m)=\mathbb P^{n-|J_m|-1}
$
to be defined by the ideal
$$
\left(
x_{(r_{m-1}+1)}, \ldots, x_{r_m}
\right),
$$
because counting its number of generators, we obtain
$$
r_m-(r_{m-1}+1)+1
=
\left( \sum_{i=1}^{i=m} (|J_1|-2) \right)
-
\left( \sum_{i=1}^{i=m-1} (|J_1|-2) 
\right)
=
(|J_m|-2)
$$
which is the  codimension of $H(J_m)$.

Finally, as we are intersecting linear subspaces in projective space condition (II) follows by the definition of the $H(J_i)$.
\end{proof}

\begin{definition}
Let $\vec w \in \mathcal D^R_n$ be an admissible weight and let $\mathcal H_{\vec w}$ be as in Lemma \ref{build}. Then the \textbf{wonderful compactification} of 
$R_{\vec w_0} \cong \mathbb P^{n-3}$ with respect to $\mathcal H_{\vec w}$ is denoted by 
$
Bl_{\vec w}R_{\vec w_0}
$.
\end{definition}

\begin{lemma}\label{lem:universaltotal}
Let $\vec w$ be an admissible weight vector in $\mathcal D^R_{n}$.  There exists a smooth variety  $\mathcal U_{\vec w}$, a flat proper morphism
$\phi_{\vec w}$, 
$$
\xymatrix{
\mathcal U_{\vec w} \ar[d]_{\phi_{\vec w}} \ar[r]^{\tau}& \mathcal U_{\vec w_0} \ar[d]^{\phi_{\vec w_0}} 
\\
Bl_{\vec  w } R_{\vec w_0} \ar[r]& R_{\vec w_0} 
}
$$
and $n$ hypersurfaces $\sigma_i(\vec w) \subset \mathcal U_{\vec w}$ such that for every $\vec s \in  Bl_{\vec w} R_{\vec w_0}$ the fiber
$\phi_{\vec w}^{-1}(\vec s)$  and the divisors 
\begin{align*}
\phi_{\vec w}^{-1}(\vec s) \cap  \sigma_i(\vec w)
& &
l_A:=\phi_{\vec w}^{-1}(\vec s) \cap \tau^{-1}\left( E_{\vec w_0} \right)
\end{align*}
define a stable sha of weight $\vec w$ 
($E_{\vec w_0}$ is defined in Lemma \ref{lemma:universal}).
\end{lemma}
\begin{proof}
Let $\vec w \in \mathcal D^R_n$ be an admissible weight and consider a 
sequence of weights $\vec \gamma_1, \ldots,\vec \gamma_{m+1}$ such that 
$\vec \gamma_1:=\vec w_0$, $\vec \gamma_{m+1} :=\vec w$, 
the weights $\vec \gamma_{i} \leq \vec  \gamma_{i+1}$ 
are adjacent to each other (see Definition \ref{def:adj}),
and $m+1$ is the minimal length of such sequences. We prove  our Lemma by induction.  The base case is proven in Lemma \ref{lemma:universal}.

Next, we describe the inductive step.  We suppose that the statement holds for $\gamma_m$. In particular, there exists a smooth variety $\calU_{\vec{\gamma_{m}}}$ with a flat proper morphism $\phi_{\vec{\gamma_{m}}}: \calU_{\vec{\gamma_{m}}} \to  Bl_{\vec{\gamma_{m}}} R_{\vec w_0}$, and $n$ hypersurfaces $\sigma_i(\vec{\gamma_{m}}) \subset \calU_{\vec{\gamma_{m}}}$ such that for every $\vec{s} \in Bl_{\vec{\gamma_{m}}} R_{\vec w_0}$ the fiber $\phi^{-1}_{\vec{\gamma_{m}}}(\vec{s})$ and the divisors $\phi^{-1}_{\vec{\gamma_{m}}}(\vec{s}) \cap \sigma_i(\vec{\gamma_{m}})$  and $l_A := \phi^{-1}_{\vec{\gamma_{m}}}(\vec{s}) \cap \tau^{-1}(E_{\vec{w_0}})$ define a stable sha of weight $\gamma_m$.

Let $W(I)$ be the wall separating $\vec \gamma_m$ and $\vec \gamma_{m+1} =\vec w$, we denote the singularity destabilized by this wall crossing by $p(I)$.

Let $\overline{H}(I)$ be the closure of the locus in $Bl_{\vec \gamma_m}R_{\vec w_0}$ parametrizing all shas $(X,D)$ with a multiple point $p(I)$, and let $S(I) \subset \mathcal U_{\gamma_m}$ be the locus supporting $p(I)$. We will show that the following diagram 
\begin{align*}
\xymatrix {
\mathcal U_{\vec w}:=
Bl_{\eta^{-1}(S(I))}
\left(
Bl_{\vec w }R_{\vec w_0}
 \times_{ \left( Bl_{\vec \gamma_{m} }R_{\vec w_0} \right) } \mathcal{U}_{\vec \gamma_m} 
\right)
\ar[r] \ar[rd]_{\phi_{\vec w}}& 
Bl_{\vec w }R_{\vec w_0}
 \times_{ \left( Bl_{\vec \gamma_{m} }R_{\vec w_0} \right) } \mathcal{U}_{\vec \gamma_m} 
 \ar[d]^{\tilde \pi} 
\ar[r]^{ \qquad{} \qquad{} \eta}
& 
\mathcal{U}_{\vec \gamma_m} 
\ar[d]^{\gamma_{\vec \gamma_m}}
\\
& 
Bl_{\vec w }R_{\vec w_0}
\ar[r]^{\rho} 
& 
Bl_{\vec \gamma_m }R_{\vec w_0}
}
\end{align*}
yields our family $\phi_{\vec w}: \mathcal U_{\vec w} \to Bl_{\vec w} R_{\vec w_0} $.

   Notice that $ S(I) \cong \overline H(I)$ 
because the projection $S(I) \to \overline H(I)$ is finite, generically one-to-one, and $\overline H(I)$  is the smooth strict transform of 
$H(I) \subset R_{\vec w_0}$ in $ Bl_{\vec \gamma_m }R_{\vec w_0} $. Therefore, the isomorphism $S(I) \cong \overline H(I)$ follows by Zariski's main theorem.

By definition of the wonderful blow up, we have that 
\begin{align*}
Bl_{\vec w }R_{\vec w_0}
=
Bl_{\overline H(I)} \left( Bl_{\vec \gamma_m }R_{ \vec w_0} \right).
\end{align*}
On another hand, by the inductive hypothesis, 
$\phi_{\vec \gamma_m}: \mathcal U_{\vec \gamma_m} \to Bl_{\vec \gamma_m} R_{\vec w_0} $ is flat. Since blowing up commutes with flat base change, we obtain
\begin{align}\label{eq:center}
Bl_{\vec w }R_{\vec w_0}
 \times_{ \left( Bl_{\vec \gamma_{m} }R_{\vec w_0} \right) } \mathcal{U}_{\vec \gamma_m} 
 \cong 
Bl_{\phi_{\vec \gamma_m}^{-1}(\overline{H}(I)) } \mathcal U_{\vec \gamma_m} 
\end{align}
which implies
\begin{align*}
Bl_{\eta^{-1}(S(I))}
\left(
Bl_{\vec w }R_{\vec w_0}
 \times_{ \left( Bl_{\vec \gamma_{m} }R_{\vec w_0} \right) } \mathcal{U}_{\vec \gamma_m} 
\right)
=
Bl_{\eta^{-1}(S(I))}
\left(
Bl_{\phi_{\vec \gamma_m}^{-1}(\overline{H}(I)) } \mathcal U_{\vec \gamma_m} 
\right).
\end{align*}

Let $E_{\rho}$ and $E_{\eta}$ be the exceptional divisors of $\rho$ and $\eta$ respectively.  Next, we describe the fiber 
$\tilde \pi^{-1}(z)$
for $z \in E_{\rho}$.  Given $y \in \overline H(I)$, the fiber 
$\phi_{\vec \gamma_m}^{-1}(y)$ is a surface $X$.  

We find,  by dimension counting,  that 
$\rho^{-1}(y) \cong \mathbb P^{(|I|-3)}$, and $\eta^{-1}(\phi_{\vec \gamma_m}^{-1}(y)) $ is a $\mathbb P^{(|I|-3)}$-fibration over $X$. Due to the fiber product construction, there is a morphism $ \eta^{-1}(\phi_{\vec \gamma_m}^{-1}(z)) \to \rho^{-1}(z) $. So $ \eta^{-1}(\phi_{\vec \gamma_m}^{-1}(z))$  is a fibration over $ \mathbb P^{(|I|-3)}$ with fibers isomorphic to $X$. 

Therefore, 
for all
$z \in E_{\rho}$ it holds that 
$\tilde \pi^{-1}(z) \cong X$, 
and
the strict transform  
$$
\{ \eta_*^{-1}(\sigma_i(\vec \gamma_m)) \; | \; i \in I \}
$$ 
of the sections $\{ \sigma_i(\vec \gamma_m) \; | \; i \in I \}$
induces a divisor  in  $\tilde \pi^{-1}(z) $   with an $(n-1)$ multiple point.  
Blowing up  $\eta^{-1}(S(I))$ generically separates those sections in 
 $\mathcal U_{\vec w}$,  because  the intersection of the hypersurfaces  $\{ \eta_*^{-1}(\sigma_i(\vec \gamma_m)) \; | \; i \in I \}$ is locally an intersection of $|I|$ hyperplanes in affine space.
Indeed, recall 
our sections are the strict transforms of $\sigma_i \subset R_{\vec w_0}$ and that  
$\mathcal U_{\vec w_0} \cong Bl_{Z} \mathbb P^{n-1}$ with $Z \cong \mathbb P^1$ and  $Z \cap \sigma_i = \emptyset$.

Finally, we describe the fibers of $\phi_{\vec w}$. The locus $\eta^{-1}(q_I) \cong \mathbb P^{|I|-3}$ intersects $\tilde \pi( z) \cong X$ at the point $x$ supporting the multiple point $q(I)$.  
The locus $S(I) \subset \mathcal U_{\vec \gamma_m}$ has dimension $(n-|I|-1)$. 
Therefore, 
$\dim(\eta^{-1}(S(I)) =(n-4)$ which implies
the divisor of the blow up
$$
\mathcal U_{\vec w} \to 
Bl_{\phi_{\vec \gamma_m}^{-1}(\overline{H}(I))} \left( \mathcal U_{\vec \gamma_m}  \right)
$$
 is a $\mathbb P^2$-fibration over $\eta^{-1}(S(I))$. So,  
$\phi_{\vec w}^{-1}(z)$ is equal to    
\begin{align}\label{eq:sur}
\mathbb P^2 \bigcup_{L=E} Bl_x \left( \tilde{\pi}^{-1}(z) \right) 
\cong
\mathbb P^2 \bigcup_{L=E} Bl_x X.
\end{align}
where $E \subset Bl_x X $ is the exceptional divisor obtained by blowing up $x$ and $L$ is a line in $\mathbb P^2$. 

The 
$\mathbb P^2$ component is a fiber of $ \mathcal U_{\vec w} \to Bl_{\phi_{\vec \gamma_m}^{-1}(p_I)} \mathcal U_{\vec \gamma_m}$,  so the strict transforms of the sections $\{ \sigma_i(\vec \gamma_m) | \; i \in I \}$ define a configuration of lines on it. Those lines do not overlap in a $|I|$-tuple point, because that is the multiple point we just separated. Therefore, the resultant pair defined by the surface in \ref{eq:sur}, and its intersection with the strict transform 
$\sigma_i(\vec w)$
of the hypersurfaces $\sigma_i(\vec \gamma_m)$ in $\mathcal U_{\vec w}$ defines a stable sha with respect to $\vec w$.
\end{proof}

In the following Lemma, we recall that  $n$ points in $\mathbb P^{n-3}$ are in  general position if there are no two of them supported in a point, no three of them contained on a line, no four of them contained in a plane, and so forth.

\begin{lemma}\label{order}
For $n \geq 5$, there  are $n$ points $q_1, \ldots, q_n$  in 
$R_{\vec w_0}$ in general position, a sequence of weights
$\vec w_k$ with $1 \leq k \leq (n-3)$, and morphisms of smooth varieties
\begin{align*}
\xymatrix{
Bl_{\vec w_{(n-3)}}R_{\vec w_0}
\ar[r]
&
\ldots \ar[r]
 & 
Bl_{ \vec w_k}R_{\vec w_0}
\ar[r]
&
\ldots \ar[r]
&
R_{\vec w_0}
}
\end{align*}
where 
\begin{itemize}
\item 
$Bl_{ \vec w_1}R_{\vec w_0}$
is the blow up of $R_{\vec w_0}$ along $q_1, \ldots, q_n$ in any order.
\item 
$Bl_{ \vec w_2}R_{\vec w_0} $ is 
the blow up of $Bl_{ \vec w_1}R_{\vec w_0} $ 
along the strict transform of lines spanned by all pairs of points 
$\{ q_i, q_j \}$, in any order
 $$ \vdots $$
\item 
$Bl_{ \vec w_{(n-3)}}R_{\vec w_0}$ is 
the blow up of 
$Bl_{ \vec w_{(n-4)}}R_{\vec w_0}$ along the 
strict transforms of the $(n-4)$-planes spanned by all $(n-3)$-tuples of the 
$q_i$, $i=1,\ldots,n$ in any order.
\end{itemize}
\end{lemma}
\begin{proof}
The wonderful blowup is by definition a sequence of iterative blow ups along the strict transforms of the elements in the building set 
$\mathcal H_{1^n}$.  The points
$q_i$ correspond to $H(I)$ with $|I|=(n-1)$, the lines spanned by the points $q_i$ correspond to  $H(J)$ with $|J|=(n-2)$, and so on.   
The order of the blow-ups can be taken to be any order of increasing dimension by \cite[Thm 1.3]{li}.
\end{proof}

\subsection{$R_{\vec w}$ is isomorphic to a wonderful compactification.}

Our aim is to show that $R_{\vec w}$ is  isomorphic to the wonderful compactification $Bl_{\vec w}\mathbb P^{n-3}$.  First we review $R_{\vec w}$ for small values of $n$.

\begin{example}\label{rmk:basecases} \leavevmode
\begin{enumerate} 
\item If $\vec w \in \mathcal D^R_3$, then $R_{\vec w}$ is a point. 
\item If  $\vec w \in \mathcal D^R_4$, then $R_{\vec w} \cong \mathbb P^1$, as $\overline{M}_{1^6}(\mathbb{P}^2, 5) \cong \overline{M}_{0,5}$.
\item If $\vec w \in \mathcal  D^R_5$, then  
$R_{\vec w} \cong Bl_{\vec w}\mathbb P^{2}$. In particular, the morphism $R_{1^5} \to R_{\vec w_0} \cong \mathbb P^2$ is the blow up of $\mathbb P^2$  at five points and the morphisms induced by wall crossings inside $\mathcal D^R_5$  are either smooth blow ups or isomorphisms. 
Indeed, it is known that $\overline{M}_{1^6}(\mathbb{P}^2, 6)$ has isolated singularities (see  \cite[Thm 4.2.4]{luxton}).  Therefore,  by the construction of $R_{\vec w}$ as in Definition \ref{def:Rn} it follows that $R_{1^5}$ is smooth.  We note that the building set $\mathcal H_{1^5}$ is described in Example \ref{bR5}, and that the smoothness of $R_{\vec w}$ follows from the smoothness of $R_{1^5}$ and Theorem \ref{main}.
\end{enumerate}
 \end{example}

\begin{theorem}
\label{thm:iso} For any choice of $n$ and $\vec w \in \mathcal D^R_n$, it holds that
$
R_{\vec w} \cong Bl_{\vec w} R_{\vec w_0},
$ and thus $R_{\vec w}$ is smooth with normal crossings boundary.
 \end{theorem}
\begin{proof} 
Our proof is by induction on the weight vector.
The base case is $R_{\vec w_0}$ which is discussed in 
Lemmas \ref{base} and \ref{lemma:universal}. 
Let $\vec v \geq \vec u$ be two adjacent weights separated by the wall $W(I)$ which destabilizes the multiple point $p(I)$.  Now consider the following diagram

\[
  \begin{tikzcd}
    Bl_{\vec v} R_{\vec w_0} \arrow{r}{f_{\vec v}} \arrow{d}{\psi_{\vec v, \vec u}} & R_{\vec v} \arrow{d}{\phi_{\vec v, \vec u}} \\
     Bl_{ \vec u} R_{\vec w_0} \arrow{r}{\cong} & R_{\vec u}
  \end{tikzcd}
\]
where the morphism $\psi_{\vec v, \vec u}$ is the blowup 
$$
  Bl_{\vec v} R_{\vec w_0} 
\cong
Bl_{\overline H(I)}
\left(
Bl_{ \vec u} R_{\vec w_0}
\right)
\to 
Bl_{ \vec u} R_{\vec w_0}
$$
 induced by the wonderful compactification, and $\phi_{\vec v, \vec u}$ is the wall crossing morphism induced by changing the weights. By induction, we assume that $Bl_{\vec u} R_{\vec w_0}
 \cong R_{\vec u}$ and thus $R_{\vec u}$ is smooth. We must now 
show that $R_{\vec v}$ is also smooth. 

By Lemma \ref{lem:universaltotal}, there is a flat family 
$\mathcal U_{\vec v} \to Bl_{\vec v}R_{\vec w_0}$ whose fibers are stable shas with respect to $\vec v$.
On the other hand, 
$\overline M_{(\vec w,1)}\left( \mathbb P^2, n+1 \right)$ is a fine moduli space by  \cite[Lemma 7.7]{wha}.
Therefore, there is a map
$f_{\vec v}: Bl_{\vec v}R_{\vec w_0} \to R_{\vec v}$.
Let 
$E_{\vec v} \subset Bl_{\vec v R_{\vec w_0}}$ 
and 
$F_{\vec v} \subset R_{\vec v}$ 
be 
the exceptional divisors of 
$\phi_{\vec v, \vec u}$  and $ \varphi_{\vec v, \vec u}$ respectively.
By construction $f_{\vec v}$ is an isomorphism 
when restricted to the open sets
$$
\left(  Bl_{\vec v}R_{\vec w_0} \right)  \setminus E_{\vec v}
\to 
R_{\vec v} \setminus F_{\vec v} 
$$
and the restriction 
$f_{\vec v}: E_{\vec v} \to F_{\vec v}$ is a finite morphism because
both exceptional divisors are $\mathbb P^{|I|-3}$ fibrations over 
$\overline H(I)$.

In particular, the above argument implies the morphism  $f_{\vec v}$ is the normalization. Therefore, since $Bl_{\vec v} R_{\vec w_0}$ is smooth, by Zariski's main theorem, it suffices to show that $R_{\vec v}$ is normal. To do so, we consider the exact sequence 
arising from normalization: $$ 0 \to \calO_{R_{\vec v}} \to f_* \calO_{Bl_{\vec v} R_{\vec w_0} } \to \delta \to 0.$$ Our goal is to prove that $\delta = 0$. If $p \in R_{\vec u}$ is a point parametrizing a 
configuration which is stable with respect to both weights $\vec v$ and $\vec u$, then the morphisms $\psi_{\vec v, \vec u}$ and $\phi_{\vec v, 
\vec u}$ are both isomorphisms, and there is nothing to prove. Therefore, we may assume that $p$ is a point which induces a blowup. 

To look at the fiber over the point $p$ we tensor by $\calO_{R_{\vec v}} / {I_p \calO_{R_{\vec v}}}$ to obtain:  $$  \calO_{R_{\vec v}} / {I_p \calO_{R_{\vec v}}} \to  (f_* \calO_{Bl_{\vec v} \bP^{n-3}}) 
 \otimes (\calO_{R_{\vec v}} / {I_p \calO_{R_{\vec v}}}) \to \delta \otimes (\calO_{R_{\vec v}} / {I_p \calO_{R_{\vec v}}}) \to 0.$$
 
 The wonderful compactification is a sequence of iterative smooth blows, so by dimension counting  the fiber of $\psi_{\vec v, \vec u}$ over $p$ is a $\bP^{|I|-3}$. Furthermore, by Lemma \ref{fibers}, the fiber of $\phi_{\vec v, \vec u}$ over $p$ is also a $\bP^{|I|-3}$. As $f_{\vec v}$ is the normalization, and both $\phi_{\vec v, \vec u}^{-1} (p)$ and $\psi_{\vec v, \vec u}^{-1}(p)$
 are scheme theoretically $\bP^{|I|-3}$, the projective spaces must be isomorphic. As the first arrow above is an isomorphism, we see that $$\delta \otimes (\calO_{R_{\vec v}} / {I_p \calO_{R_{\vec v}}})  = 0.$$ As this is true for all $p \in R_{\vec v}$, we see that $\delta = 0$, and thus $R_{\vec v}$ is normal. 
 \end{proof}

\subsection{Consequences of the blow up construction}

\begin{theorem} \label{main}
There is a birational projective morphism
$
\R_{\vec w} \to \R_{\vec w_{0}} \cong \mathbb{P}^{n-3}
$  
that can be understood as a sequence of smooth blowups.
 In particular, the morphism
$\R_{1^n}  \to \mathbb{P}^{n-3}$ can be understood as completing the steps descried in Lemma \ref{order}.
\end{theorem}
\begin{proof}
The theorem follows from  Theorem \ref{thm:iso}.
\end{proof}

\begin{lemma}
\label{hassett}\cite{hassett2003moduli}
Let  $\vec \alpha=(\alpha_1, \ldots, \alpha_n)$ be a set of weights where 
\begin{align*}
\alpha_1 =1,
& &
\alpha_2 = 1-\frac{(n-2)}{n-1}+\frac{1}{2(n-1)},
& &
\alpha_3=\ldots =\alpha_n=\frac{1}{n-1}.
\end{align*}
Then $\overline M_{0, \vec \alpha}=\mathbb P^{n-3}$. 
Let 
$\delta_I \subset \bP^{n-3}$ be the locus parametrizing configurations of $n$ points in $\bP^1$ such that
$\{ p_{i_1}=\ldots = p_{i_s} \; | \; i_k \in I \}$.
Then,  for every 
$\vec w > \vec \alpha$, it follows that 
$
\overline M_{0, \vec w}
$
is the wonderful compactification of $\mathbb P^{n-3}$ with respect to the building set
$$
\mathcal S_{\vec w} := \{ 
\bP^{(n-|I|)-2} \cong
 \delta_I \subset \bP^{n-3} \; | \; \sum_{i \in I} w_i > 1, \;
 I \subset \{ 2, \ldots, n \}, \;
2 \leq |I| \leq (n-2)
\}.
$$
\end{lemma}
\begin{proof}
The existence of a set of weights 
$\vec \alpha$ such that $\overline M_{0, \vec \alpha} \cong \mathbb P^{n-3}$ is well-known (see \cite[Sec 6.2]{hassett2003moduli}). The condition 
$\vec w > \vec \alpha$ guarantees the existence of a morphism 
$\overline M_{0, \vec w} \to \overline M_{0, \vec \alpha} $
 (see \cite[Thm 4.1]{hassett2003moduli}).
 The set $\mathcal S_{\vec w}$ is the locus
in $\mathbb{P}^{n-3}$ that becomes unstable with respect to the weights $\vec w$.  In particular, the condition $1 \not\in I$ is necessary for $\delta_I \subset \mathbb P^{n-3}$, otherwise $\delta_I$ is unstable with respect to $\vec \alpha$.
\end{proof}

\begin{corollary}\label{genw}
Given  a set of weights $\vec w =(1, w_2, \ldots, w_{n})$, 
there is a morphism
$
\R_{\vec w} \xrightarrow{} \overline M_{0, \vec w}
$
which can be interpreted as a continuation of a blow up construction 
$\overline M_{0,\vec w} \to \mathbb{P}^{n-3}$. 
\end{corollary}
\begin{proof}
The weights of $l_A$ and $l_1$ are one by construction, then we can define the morphism
$
\psi: R_{\vec w}  \to \overline M_{0,\vec w}
$
by intersecting the broken lines $\{ l_A, l_2, \ldots, l_n \}$ with $l_1$. That is
$$
\left(
X, l_A+\sum_{k=1}^{n} w_{k}l_k
\right)
\longrightarrow
\left(
l_1,  (l_A+\sum_{k=2}^{n} w_{k}l_k)\bigg|_{l_1}
\right).
$$
The morphism is well defined by adjunction.
Notice that the set
$
\left\lbrace
H(I) \in \mathcal H_{\vec w} \; | \; 1 \in I
\right\rbrace
$
is isomorphic to  $\mathcal S_{\vec w}$ as in Lemma \ref{hassett} above.
Indeed,
for an index set $I \subset \{1, \ldots n \}$ such that $1 \in I$, it holds that
$$
\sum_{i \in I} w_{i} > 2 \Longleftrightarrow 
\sum_{i \in I \setminus {1}} w_i > 1.
$$
Moreover, 
$
\bP^{(n-|I-1|)-2}
\cong
\delta_{I \setminus 1}  \cong H(I) \cong \bP^{(n-|I|)-1}
$
by Lemma 
\ref{bij},  and if $I$ and $K$ are indices 
containing 1, then 
$\delta_{I\setminus 1} \cap \delta_{K \setminus 1} \neq \emptyset$ if and only if 
$H(I) \cap H(K) \neq \emptyset$. Finally, we use that 
$
\mathbb P^{n-3} \cong R_{\vec \alpha} \cong \overline M_{0,\vec \alpha}
$ to identify these sets.

By \cite[Thm 1.3.ii]{li}, the wonderful blowup does not change if we rearrange the elements of $\mathcal H_{\vec w}$ so that
the first $k$ terms form a building set for any $1\leq  k \leq n$.  Therefore, by Theorem \ref{thm:iso}, we have
\begin{align*}
\R_{\vec w}
=
Bl_{\mathcal H_{\vec w}} \left( \bP^{n-3} \right) 
=
Bl_{\mathcal H_{\vec w} \setminus \mathcal S_{\vec w}  }   \left(  Bl_{\mathcal S_{\vec w}  }  \bP^{n-3} \right)
=
Bl_{\mathcal H_{\vec w} \setminus \mathcal S_{\vec w}  
}  \left(  \overline M_{0,\vec w} \right)
\end{align*}
where $Bl_{\mathcal H_{\beta} \setminus S_{\beta} }$ denotes the blow up along the strict transform of the elements in  the set 
$\mathcal H_{\beta} \setminus S_{\beta}$.

The description in the statement of our result follows by comparing Lemma \ref{order} with the blow up construction of $\overline M_{0,n}$ outlined  in the introduction.\end{proof}

We now show that there do not exist weights $\vec w$ so that $R_{\vec w} \cong \overline{M}_{0,n}$. 

\begin{proof}[Proof of Theorem \ref{exclu}]
Let $\vec w \in \mathcal D^R_n$,  we will show that $\mathcal H_{\vec w}$ cannot be equal to the locus 
$S_{\vec w}$ required to construct $\overline M_{0,n}$
as described in \cite[Sec 6.2]{hassett2003moduli}. 
If we suppose otherwise, then $\vec w$ destabilizes $(n-1)$ points and all the linear subspaces spanned by them while the $n$th point is stable with respect to $\vec w$. 
In other words, 
let $H(I_k) \in \mathcal H_{\vec w}$, where $|I_k| = (n-1)$ for $ k = 1, ..., n-1$ and  $H(I_{n}) \not\in \mathcal H_{w}$ where $|I_n|=(n-1)$.  The existence of $\vec w$ is equivalent to the existence of a solution for the following system of inequalities. 
\begin{align}\label{eq:1}
w_{i_1}+w_{i_2}+w_{i_3} >  2
& 
\qquad{}
\forall \; 
\{ i_1, i_2, i_3\} \subset I_{k} 
&
\\
\label{eq:2}
 \sum_{i \in I_{n}} w_i \leq 2
,&  
\qquad{}
0 < w_i \leq 1.
&
\end{align} 
The inequality \eqref{eq:1} is associated to destabilizing the $(n-2)$-planes generated by $H(I_k)$ with $1 \leq i \leq n$. The inequality 
\eqref{eq:2} follows because $H(I_n)$ is stable with respect to $\vec w$. 
Without loss of generality, we set 
$I_{n}=\{2, \ldots, n \}$.  Since $|I_k|=(n-1)$,  there is at least one $I_k$ such that  $I_k \cap I_n $ has at least three  distinct elements  $i_1, i_2, i_3$ and so the inequality \eqref{eq:1} for these three elements contradicts \eqref{eq:2}.
\end{proof}

\section{$R_{1^n}$ as a Non-reductive Chow Quotient}\label{sec:chow}

In this section, we discuss the proof of Theorem  \ref{ChowThm}.   An important step of the proof is based on the fact that the dual graphs of the pairs parametrized by $R_{\vec{w}}$ are always rooted trees,  with the root vertex corresponding to the component containing the line $l_A$. To keep track of the lines $l_i$,  we mark the vertices corresponding to the last component containing the broken line $l_i$.  \\
\begin{figure}[h!]
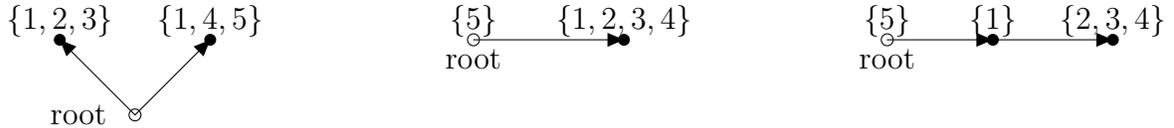

\baselineskip=10pt
\hsize=6.3truein
\vsize=8.7truein
\tikzpicture[line cap=round,line join=round,>=triangle 45,x=1.0cm,y=1.0cm, scale=0.5]
\draw [->] (-6,10) -- (-2,10);
\draw [->] (7.81,10.0) -- (11,10);
\draw [->] (5,10) -- (7.81,10.0);

\draw [->] (-15,8) -- (-13,10);
\draw [->] (-15,8) -- (-17,10);
\draw [color=black] (-15,8) circle (4.5pt);
\fill [color=black] (-17,10) circle (4.5pt);
\fill [color=black] (-13,10) circle (4.5pt);
\draw (-17,10.5) node {$\{ 1,2,3 \}$};
\draw (-13,10.5) node {$\{ 1,4,5 \}$};
\draw (-16.5,8) node {root};

\draw (-6,10.5) node {$\{ 5 \}$};
\draw (-6,9.5) node {root};
\draw (-2,10.5) node {$\{ 1,2,3,4 \}$};
\draw [color=black] (-6,10) circle (4.5pt);
\fill [color=black] (-2,10) circle (4.5pt);

\draw (5,10.5) node {$\{5 \}$};
\draw (7.81,10.5) node {$\{ 1 \}$};
\draw (11,10.5) node {$\{ 2,3,4 \}$};
\draw (5,9.5) node {root};
\fill [color=black] (11,10) circle (4.5pt);
\fill [color=black] (7.81,10.0) circle (4.5pt);
\draw [color=black] (5,10) circle (4.5pt);

\endtikzpicture
\caption{Left to right: Dual graphs associated to the last sha of Fig.
\ref{examples} and last 2 of Fig. \ref{examples2} resp.}
\end{figure}

We highlight that there is a configuration space known as $T_{d,n}$ which  generalizes $\overline{M}_{0,n}$ (see \cite{chen2009pointed}), and is a non reductive Chow quotient  under the same group \cite{pn}. The objects parametrized by $T_{d,n}$ are known as stable rooted trees, and are the union of surfaces $X \cong Bl_m\bP^2$ , as in our space, but with markings given by points rather than lines.

\begin{remark}\label{rmk:kapranov} We recall Kapranov's construction of  $\overline{M}_{0,n}$ as a Chow quotient (see \cite{chow}). Given a collection of  $n$ generic points  $p_i $ in $\bP^1$, we  consider the cycle associated to the closure of the orbit:
$
\overline{ \SL_3\cdot (p_1, \ldots, p_n)}
\subset (\bP^1)^n.
$
Varying the points, we obtain cycles parametrized by an open locus in the appropriate Chow variety. Taking the closure of this open set, we obtain the Chow quotient $(\bP^1)^n \ChowQ \SL_2$ which is isomorphic to  $\overline{M}_{0,n}$. \end{remark}

We fix our line $l_A$ once and for all, and  denote by $\hat \bP^2$ the dual projective space.  The lines $\{ l_1, \ldots l_n  \}$ are parametrized by points $p_1, ..., p_n \in (\hat \bP^2)^{n}$.   
Let $G \subset \operatorname{SL}(2, \mathbb C)$ be the group 
acting on $\bP^2$  that fixes the line $l_A \subset \bP^2$ pointwise. Then
$G \iso \mathbb G_m \rtimes \mathbb G_a^2$, $\dim(G)=3$, and if $l_A:=(x_0=0)$,  the group consists of elements of the form:
\begin{align*}
G = 
\begin{pmatrix}
t^{-2}&      0 & 0     \\
s_0 & t & 0 \\
s_1 & 0 & t
  \end{pmatrix}
\end{align*}
Given a point $p_i=[a_0:a_1:a_2] \in \hat \bP^2$, the  line associated to it by projective duality can be written as  $l(\vec x) := \left( p_i \cdot x=0 \right)$.  
Then we have $l(g\cdot x)= (p_i \cdot g )(x)$  from which 
we obtain  the following action of $G$ on $\hat \bP^2$.

\begin{definition}
Let $g \in G$ be as above, then we define the action on $\hat \bP^2$ as
$$
g \cdot [a_0:a_1:a_2]
:=
[t^{-3}a_0+\frac{s_0}{t}a_1+\frac{s_1}{t}a_2: a_2:a_3]
$$
After acting with the group, the line $l(x)=(a_0x_0+a_1x_1+a_2x_2=0)$ 
becomes
$$
\left( t^{-3}a_0+\frac{s_0}{t}a_1+\frac{s_1}{t}a_2 \right) x_0+ a_1x_1+a_2x_2
$$
In particular, the intersection point $l(x) \cap (x_0=0)$ is invariant under the action of $G$.
\end{definition}

Inside $(\hat \bP^2)^{n}$, we define the loci
$$
U(q):=\{(p_1, \ldots, p_n) \in (\hat \bP^2)^n \; | \; l_i \cap l_A
\text{ are fixed with equivalence class  $q \in M_{0,n}$.} \} 
$$
Notice that   $\dim(U_{n}(q))=n$. We select once and for all a connected component of  the closure of $U(q_n)$ and we denote it, by abuse of notation, 
as $\overline U(q_n)$. In particular, we fix an intersection $\{ l_i \cap l_A \}$ once and for all for the rest of this chapter, so we omit it after here and just write  $\overline U$.

\begin{proposition}\label{prop:openquot}
The Chow quotient 
$ \overline U/ \! \!/_{Ch}G$ 
is birational to  
$R_{1^n}$.
\end{proposition}
\begin{proof}
By shrinking if necessary, we can find an open subet $U' \subset U$  contained in a  $G$-invariant open locus in $(\hat \bP^2)^n$, so that there is a natural map $\psi : U'  \rightarrow R_{1^n}$. Furthermore, the $G$-action fixes 
the line $l_A$ pointwise, and thus fixes $l_i \cap l_A$. As a result, all configurations in the orbit $G \cdot l_i$ are isomorphic as line arrangements in $\bP^2$, and thus are equivalent in $R_{1^n}$. Therefore, $\psi$ is $G$-invariant and induces a 
morphism $\overline{\psi}: U' / G \to R_{1^n}$. 
This morphism is injective on an open set in $R_{1^n}$, because if generic $p, p' \in U'$  satisfy
$\overline{\psi}(p)=\overline{\psi}(p')$, then there is a $g \in SL(3, \mathbb C)$ such that 
$ g \cdot p=p'$. 
This last equality implies $g$ fixes the line $l_A$ as well as all of the intersections $l_i \cap l_A$, and so $g \in G$ and $p$ and $p'$ are in the same $G$-orbit. The map $\overline{\psi}$ is dominant, because for a generic isomorphic class of lines parametrized by $\R_{1^n}$,  we can choose a representative where $l_A$ and $l_i \cap l_A$ are as in the beginning of this section, and that representative is parametrized by $U'$. 
\end{proof}

Next, we show that the birational map 
$\rho : R_{1^n} \dashrightarrow \overline{U} \ChowQ G$ 
is a regular morphism.   This is done by associating a cycle to each sha $X$ parametrized by 
$\overline R_{1^n}$. 
We recall that each component $X_v$ of $X$ is either $\bP^2$, the blow up of $\bP^2$ at finite number of points, or $\bP^1 \times \bP^1$ (see Proposition \ref{sing}), and that there is a \emph{contraction morphism} 
$\varphi_{v} : X \rightarrow \PP^2$ that contracts $X_v$ to $\bP^2$ while also contracting all other components. 
 For each $v \in I$, the contraction morphism induces a
 line 
arrangement $\varphi_{v}(X)$
defined  up to choice of coordinates. We always select a representative which, by an abuse of notation, we denote by $\varphi_{v}(X)$, so  $l_A:=(x_0=0)$ and the points $l_A \cap l_i$ are the same as the ones used to define $U$.

\begin{definition}\label{def:compconfig}
Fix a closed point of $R_{1^n}$ parametrizing the sha $X = \cup_{v \in I} X_v$. 
The \textbf{configuration cycle} $Z(X)$ is: 
$$
Z(X) : = \displaystyle\sum_{v \in I} 
\overline{ G \cdot \varphi_v(X) } \subsetneq ( \hat \bP^2)^n.
$$
\end{definition}
\noindent
We must show that these configuration cycles  all have the same dimension and homology class.
Let $\vec m:=\{ m_1, \ldots ,m_n \}$ be a set of integers
such that  $\sum_{i=1}^n m_i = 3$ and $0 \leq m_i \leq 2$.
By the K\"unneth formula, a basis for the  
homology  in $(\hat \bP^2)^n$ is  $[\hat \PP^{m_1}]\otimes\cdots\otimes[\hat \PP^{m_n}]$. 
Let $\mathbb L_{\vec m} := L_1\times \cdots \times L_n$ be a collection of generic linear subspaces $L_i \subseteq \hat \PP^2$ of codimension $m_i$. 
The homology class of the orbit $\overline{G \cdot p}$ is 
$$
[\overline G \cdot p]= 
\sum_{\vec m} c_{\vec m}\left( [\PP^{m_1}]\otimes\cdots\otimes[\PP^{m_n}] \right)
$$
where  $\left( \overline{ G \cdot p }  \right) \cdot \mathbb L_{\vec m}$ is the intersection of the orbit $\overline{ G \cdot p } $ with the generic linear subspaces $\mathbb L_{\vec m}$.
\begin{proposition}\label{prop:homclass}
Let $\vec m$ be as above
and $X = \cup_{v\in I} X_v$,  then the homology class $[Z(X)]$ 
of the cycle $Z(X)$ is 
\begin{align}\label{ZX}
[Z(X)]:=
\sum_{\vec m={m_1, \ldots, m_n}}
\left(
\sum_{v \in I} \overline{ G \cdot \varphi_v(X) }  \cdot \mathbb L_{\vec m}
\right)
\left( [\PP^{m_1}]\otimes\cdots\otimes[\PP^{m_n}] \right)
\end{align}
\noindent
In particular, if $X$ is a generic point of $R_{1^n}$ (i.e. $X$ is supported on a single $\bP^2$). Then
$$
[Z(X)]= 
\sum_{\vec m} c_{\vec m}\left( [\PP^{m_1}]\otimes\cdots\otimes[\PP^{m_n}] \right)
$$
where $c_{\vec m}$ is either 0 or 1.  
\end{proposition}
\begin{proof}
The result follows verbatim from the analogous \cite[Proposition 2.1.7]{chow}. The main idea is as follows: let $p_i \in \hat \bP^2$ be the points parametrizing the lines $l_i$ in $\varphi_v(X)$. Then,  $c_{\vec m}=1$ if and only if there is a unique $g\in G \subset SL(3, \mathbb C)$ such that $g\cdot p_i \in L_i$ for all $1 \le i \le n$; and $c_{\vec m}$ is zero if there is no such as $g \in G$. For generic $X$ those are the only cases, so we only have those coefficients.
\end{proof}
It will turn out that we only need to calculate the homology of the cycles associated to the maximal degenerations parametrized by 
$R_{1^n}$.
\begin{lemma}\label{lemma:maxdegen}
A closed point $X =\cup_{v \in I}X_v$ in $R_{1^n}$ is \emph{maximally degenerate}, that is it lies on a minimal (i.e., deepest) stratum of the boundary stratification, if and only if the configuration of lines $\varphi_v(X_v)$ has exactly three lines $l_i$ with $1 \leq i \leq n$ in general position for every $v\in I$, not including $l_A$ or its image.
\end{lemma}
\begin{proof}
Recall that the group $G$ is three dimensional. If $\varphi_v(X_v)$ has more than three lines, not including $l_A$ or its image, in general position, then $X_v$ has moduli larger than zero, and it can be degenerated further. 
\end{proof}

\begin{proposition}\label{prop:maxdegenclass}
If the sha $X \in R_{1^n}$ is maximally degenerated, then the homology class of $Z(X)$  has all coefficients $c_{\vec m}$ equal to $1$ if and only if for all $m_i \in \vec m$ we have that $m_i \neq 2$. 
\end{proposition}
\begin{proof}
First we show the $(\Rightarrow)$ direction by proving the contrapositive. Suppose that there is an $m_i \in \vec m$ such that $m_i=2$. Then  
we claim that for each component $X_v$ of $X$, we have that  
$\varphi_v(X) \cdot \mathbb L_{\vec m'} =0$.  Indeed, $m_i=2$ implies that there is a generic linear subspace 
$L_i  \in \mathbb L_{\vec m'}$ such that
$L_i \cong \bP_i^0 \subset \hat \bP^2$ is a point.  By projective  duality, we obtain a line $\bP_i^1$ in $\bP^2$ that has generic intersection with $l_A$. 
However, there does not exist a $g \in G$ such that $g \cdot l_i=\mathbb P^1_i$
, because  this would imply that both $l_i$ and the $\bP^1_i$ would  intersect $l_A$ at the same point. This is impossible given our action of $G$, because $G$ restricts to the identity in $l_A$.

Next, we show the $(\Leftarrow)$ direction.
We divide the set of lines 
in $\varphi_v(X) \cong \bP^2$
into sets $I_i(v)$ and  $I_A(v)$, where $I_i(v)$ denotes the set of lines associated to the  
the multiple points $p(I_i(v)) \in \varphi_v(X)$ (i.e. points of multiplicity $\geq 3$), and the set $I_A(v)$ denotes the lines overlapping with $l_A$. 
By construction, $I_i(v) \cap I_A(v) =\emptyset$. However, the sets $I_i(v)$ are not necessarily disjoint, as lines can support more than one multiple point.   Of course, if the configuration only has double points, then $I_i(v) =\emptyset$. 
We define the numbers 
$m_i(v):=  \sum_{k \in I_i(v)}m_k$ and 
$m_A(v) := \sum_{k \in I_A(v)}m_k$. If $I_i(v) =\emptyset$, then we take $m_i(v):=0$, and similarly for
$I_A(v)$. We make the following claim.

\begin{claim}\label{lemma3} $\varphi_v(X) \cdot \mathbb L_{\vec m} > 0 \iff m_A(v) =0$, $m_i(v) \leq 2$, and $m_k \leq 1$ for all $i$ and $m_k \in \vec m$.  \end{claim}
\begin{proof}[Proof of Claim \ref{lemma3}] We start with the $(\Rightarrow)$ direction. If $m_A(v)>  0$, then we have  a generic line $L_i \subset \hat \bP^2$
with $i \in I_A(v)$, and thus a generic point $\bP^0_i \subset \bP^2$ in the dual space.  We must find a $g \in G $ such that 
$\bP^0 \in g(l_i)$ for a line $l_i$ that overlaps with $l_A$. This is impossible, because $G$ does not move $l_A$, and so $\varphi_v(X) \cdot \mathbb L_{\vec m} =0$. 

Next, suppose that $m_i(v)=3$.
By the previous argument, we know that if $m_i=2$, then $\varphi_v(X) \cdot \mathbb L_{\vec m} =0$. Then  up to relabelling, we  can assume that  $m_1=m_2=m_3=1$ and that  $\{ 1,2,3\} \subset I_1(v)$.  The generic lines $L_1, L_2, L_3$ in $\hat \bP^2$ induce three generic points $\bP^0_s$ in $\bP^2$. We need to find a $g \in G$ such that the points $\bP^0_s \in g \cdot l_s$ for $s \in {1,2,3}$. Again, this is impossible by the geometry of the problem. Indeed, recall that the intersection points of the lines $l_s \cap l_A$ are fixed. We can find two lines passing through 
$\bP^0_1$ and $\bP^0_2$, but those two lines will intersect at $p(I_1)$, and thus determine the position of all the other lines in $I_1(v)$. Therefore, a generic $\bP^0_3$ will not be contained in $g \cdot l_3$, and therefore $\varphi_v(X) \cdot \mathbb L_{\vec m} =0$. \\ 
\indent We continue with the $(\Leftarrow)$ direction of the claim.
There are three $L_s$ of codimension one, 
and we can suppose that  
$s \in \{ 1,2,3\}$. By duality, they induce  three points in general position in $\bP^2$. The statement follows because we can find three lines that pass through these three points as along as the lines are in general position. This holds, because $m_i(v) \leq 2$ implies that $\{1,2,3\}$ is not a subset of $I_i(v)$ for any $i$. \end{proof}

By Expression \ref{ZX} in Proposition \ref{prop:homclass}, our statement follows if we prove that for a given $\vec m$, and any sha $X=\cup_v X_v$ parametrized by $R_{1^n}$, there exists a unique component $X_v$ satisfying the criteria of Claim \ref{lemma3}. 
The following argument uses the description of the dual graph of the  $X$, which is a rooted tree by Lemma \ref{sing}.
We start with the root component $X_0$. There is no line coinciding with $l_A$ in $\varphi_0(X)$, and so $m_A(0)=0$. 
Thus there are two options: \begin{enumerate} \item Either $m_i(0) \leq 2$ for all $i$, or \item there exists an $i$ such that $m_i(0)=3$.  \end{enumerate} 

\textbf{Case (1):} If  $m_i(0) \leq 2$ for all $i$, then $\varphi_0(X) \cdot \mathbb L > 0$. To show uniqueness, 
recall that  $\sum_{i = 1,2,3} m_i = 3$, and that $m_i(v) \leq 2$ for all $i$. 
Therefore, the root has at least two branches,  and each
of those branches has at least one index $i_0$ such that $m_{i_0}=1$. Then $I_A(v)$ contains at least one of these indices for every other component $v \neq 0$, because at least one those branches is contracted with its line $i_0$ that overlaps with $l_A$. Therefore, $m_A(v) >0$, and thus $\varphi_v(X) \cdot \mathbb L_{\vec m} = 0$.

\textbf{Case (2):} If there exists an $i$ such that $m_i(0)=3$, then 
 $\varphi_0(X) \cdot \mathbb L_{\vec m} =0$. Thus we may suppose after relabeling, that $m_1(0)=3$, and that $m_1=m_2=m_3=1$ with $\{1,2,3 \} \subset I_1(0)$. This means that there is a branch starting from the root which contains the lines $\{1,2,3 \}$. Let $X_{v'}$ be the component in that branch that 
intersects with the rooted component.  We claim that $m_A(v')=0$, because 
$I_A(v')$ denotes the set of lines in the other branches which are not in $I_1$. Those indices do not include $\{1,2,3\}$, and these indices are the only
ones of  weight one. Thus, we have two options:\begin{enumerate}
\item If $m_i(v') \leq 2$ for every $i$, then 
we have that $\varphi_{v'}(X) \cdot \mathbb L_{\vec m} > 0$. Uniqueness follows by same argument used above. There are at least two  branches starting from $v'$ with an index $j$ such that $m_j=1$. Any other $\varphi_v(X)$ will 
contain that index in $I_A(v)$,  and so $\varphi_v(X) \cdot \mathbb L_{\vec m} = 0$.
 \item If $m_i(v') =3$ for some $i$, then there is a branch starting  from the vertex $v'$ that contains the lines $\{1,2,3\}$.
\end{enumerate}
 In the last case, we repeat the above argument with the surface $X_{v''}$ that intersects $v'$ and belongs to the branch containing the lines with indices $\{1,2,3\}$. Since for any sha the tree is finite,  one of the next two things must happen.
\begin{enumerate}
\item We find a component $\hat v$ such that $\varphi_{\hat v}(X) \cdot \mathbb L_{\vec m}> 0$.
It is unique by above arguments, or
\item we arrive to the last vertex of a branch that we call $v_f$.
\end{enumerate}
In the last case, we have at most three lines in general position on 
$X_{v_f}$, 
because by assupmtion $X$ is maximally degenerated;
and there are no multiple points. Following our labeling, those lines are precisely $\{1,2,3\}$,  and so $m_A(v_f)=m_i(v_f)=0$, and 
$\varphi_{v_f}(X) \cdot \mathbb L_{\vec m} > 0$.
\end{proof} 

Next, we extend the birational map 
$\rho : R_{1^n} \dashrightarrow \overline{U} \ChowQ G$ to a regular morphism. Note that there exists at most one extension, since the image is dense and the Chow variety is separated. Furthermore, the image of an extension as above is contained in $ \overline {U} \sslash_{Ch} G$, since this Chow quotient is closed in the Chow variety.  We begin with a crucial lemma. 

\begin{definition}\cite[Definition 7.2]{gg} Let $(A,\mathfrak{m})$ be a DVR with residue field $k$ and fraction field $K$, and let $Y$ be a proper scheme. By the valuative criterion, any map $g : \spec K \to Y$ extends to a map $g : \spec A \to Y$ . We write $\lim g$ for the point $g(\mathfrak m) \in Y$ . \end{definition}

\begin{lemma}\cite[Theorem 7.3]{gg}\label{DVR}
Suppose $X_1$, $X_2$ are proper schemes over a noetherian scheme $S$ with $X_1$ normal. Let $ U \subset X_1$ be an open dense set 
and $f :U \to X_2$ an $S$-morphism. Then $f$ extends to an $S$-morphism $\hat{f}:X_1 \to X_2$ if and only if for any DVR $(K, \mathfrak{m})$ and any morphism 
$g:\text{Spec}(K) \to U$, the point $\text{lim}fg$ of $X_2$ is uniquely determined by the point $\text{lim}\;  g$ of $X_1$. 
\end{lemma}

Our argument for the following result follows the same structure as the one used for the proof of $\overline M_{0,n}$  (see \cite[Thm 1.1]{noahrnc}), 
and $T_{d,n}$ (see \cite[Sec. 4.3]{pn}).
\begin{proposition}
There is a morphism $\rho : R_{1^n} \rightarrow (\hat \PP^2)^n\ChowQ G$ that associates to each closed point $X =\cup_{v \in I} X_v$ of 
$R_{1^n}$ a cycle with homology class 
\begin{align*}
\sum_{\vec m:=(m_1, \ldots,m_n)}\left( [\PP^{m_1}]\otimes\cdots\otimes[\PP^{m_n}]\right)
& &
0 \leq m_i \leq 1, & & \sum_{i=1}^n m_i=3.
\end{align*}
\end{proposition}
\begin{proof}

Consider a flat proper 1-parameter family  $X_{\Delta} \rightarrow \Delta$ where the generic fiber $X_t$ is a sha 
parametrized by the interior $R^\circ_{1^n}$. Then $X_t$ is supported in $\bP^2$ without any multiple point of multiplicity larger than two, and the central 
fiber $X_{\bC} \to \spec{\bC}$ is an arbitrary 
closed point of $R_{1^n}$. The cycle $[Z(X_t)]$ associated to a generic fiber in $X_t$ is three dimensional, and its homology class is $\delta$ (see Proposition \ref{prop:homclass}). Therefore, we have a 1-parameter family of cycles whose limit in the Chow variety we denote as  $\lim_{t \to 0} [Z(X_t)]$.
By Proposition \ref{prop:openquot} and Lemma \ref{DVR}, the existence of the morphism then
follows if we show that $\lim_{t \to 0} [Z(X_t)]$ is uniquely determined by $X_{\mathbb C}$. 
It suffices to show that:
\begin{align}\label{equal}
\lim_{t \to 0} [Z(X_t)]=[Z(X_{\mathbb C})]
\end{align}
where $[Z(X_{\mathbb C})]$ is equal to the cycle defined in Proposition \ref{prop:homclass}.

First we show that $Z(X_{\mathbb C}) \subseteq \lim_{t \to 0} Z(X_t)$ as subvarieties of  $(\hat \PP^2)^n$. 
Since $X_{\mathbb C}=\cup_{v} X_v$, by  definition of $Z(X_{\mathbb C})$, our claim follows if 
for every component $X_v$ of $X_{\mathbb C}$, we have that:
$$
\varphi_v(X_{\mathbb C}) \subset \lim_{t \to 0} Z(X_t)
\subset (\hat \bP^2)^n
$$ 
By construction $\lim_{t \to 0} Z(X_t)$ is closed and $G$-invariant. 
Therefore, our claim follows if 
$\varphi_v$ maps the points 
$(p_{1_0}, \ldots p_{n_0}) \in (\hat \bP^2)^n$ 
associated to the lines in $\varphi_v(X_{\mathbb C})$,  into 
$$\lim_{t \to 0} Z(X_t) \subset (\hat \bP^2)^n.$$
We recall that in general for shas, the contraction morphism $\varphi_v:X_{\mathbb C} \to \bP^2$ is induced by a line bundle $L_v$ that satisfies $h^i(X, L_v)=0$ for all $1 \geq i$, since $\varphi_v$ is degree 1 on the $X_v$ component and degree 0 elsewhere.
Then, by Grauert's Theorem (see Corollary III.12.9 of \cite{h}), the morphism $\varphi_v$ lifts to a morphism
from the central fiber to our 1-parameter family $X_{\Delta}$. 
Let $\varphi_v: X_{\Delta}\to  (\hat \bP^2)^n$ be that lift. For $t \neq 0$, the map 
$\varphi_v$ sends the points $p_{i_t} \in (\hat \bP^2)^n$ associated to the lines in  $\varphi_v(X_{t})$ to  $Z(X_t)$, and the morphism $\varphi_v$ is continuous. Then, $\varphi_v(X_{\mathbb C}) \subset \lim_{t \to 0}Z(X_t)$; and we have
\begin{align}\label{lhomo}
[Z(X_{\mathbb C})] \leq \lim_{t \to 0}[Z(X_t)].
\end{align}

Next, we show the equality. By Proposition \ref{prop:homclass}, we know that the homology class of the generic orbit has coeficients equal to either 0 or 1.
By the argument in the proof of Proposition \ref{prop:maxdegenclass}, we conclude that the  homology class of the generic orbit has coefficient $c_{\vec m}=0$ 
if there is an $m_i \in \vec m$ such that $m_i=2$. Indeed, it will induce a generic line $\bP^1_i \subset \bP^2$; and we cannot move any lines $l_i$ to such a line because the intersections  $l_i \cap l_A$ are fixed.   
On the other hand, for $t_0 \neq 0$ we see that:
\begin{align}\label{limitgen}
\lim_{t \to 0}[Z(X_t)] =[Z(X_{t_0})]
\end{align}
because we are taking the limit inside a Chow variety. Consequently, the homology class of the limit is the same as the homology class of the generic fiber

Expressions \ref{lhomo} and \ref{limitgen}
imply that the coefficients  $c_{\vec m}^{gen}$ in the homology class of the generic element $Z(X_{t_0})$ are necessarily  larger than or equal to the coefficients $c_{\vec m}^{0}$ associated to the central fiber $Z(X_{\mathbb C})$ . Therefore we have the following inequality 
\begin{align}
1 \leq c_{\vec m}^{0} \leq c_{\vec m}^{gen} \leq 1, 
\end{align}
The left inequality follows by Proposition \ref{prop:maxdegenclass}
and because the homology class only decreases whenever degenerating, as seen in (\ref{lhomo}).  The right inequality follows from
Proposition \ref{prop:homclass}. We conclude that there is a morphism $\rho: R_{1^n} \rightarrow (\hat \PP^2)^n\ChowQ G$.  \end{proof}
Finally, we prove  that $R_{1^n}$ is isomorphic to the normalization of our Chow quotient. 
\begin{theorem}\label{chowt}
Let $\overline {U}^n/\!\!/_{Ch} G$ be the normalization of the Chow quotient,
and let $\rho^n$ be the morphism obtained from the Stein factorization of $\rho$.
Then the morphism $$\rho^n : R_{1^n} \rightarrow  \overline {U}^n/\!\!/_{Ch} G $$  is an isomorphism.
\end{theorem}
\begin{proof}
We use the Zariski's Main Theorem which asserts that a quasi-finite birational morphism to a normal, Noetherian scheme is an open immersion.  $R_{1^n}$ is normal, and 
our morphism $\rho $ factors through the normalization of the Chow quotient. Then, $\rho^n$ is surjective and birational; and the crux of the result is to prove that $\rho$ is quasi-finite. 
By Proposition \ref{prop:openquot}, we already know the map $\rho$ is injective on the interior $R^{\circ}_{1^n}$; and we observe that no point of the boundary divisor in  $R_{1^n}$ can be sent to the same cycle as a point of the open stratum, since the image of the latter is an irreducible cycle whereas the image of the former is not.
Therefore, we only need to show that the restriction of $\rho$ to the boundary in $R_{1^n}$ is quasi-finite. The boundary is the union of a finite number of divisors, and so it will be enough to show our claim for a single component $D_I$ of the boundary. 
The general point of the divisor $D_I$ parametrizes a sha 
$X = \bP^2 \cup \bl$, where $\bl$ contains the line $l_A$. 
For example, the second sha in Figure \ref{examples} is parametrized by $D_{2345}$. 
The morphism $\rho$ sends $X$ to the union of the two cycles: 
$$
\overline{ G \cdot \varphi_0(X)} \cup \overline{ G \cdot \varphi_1(X)}
$$
If another sha $\tilde X$ parametrized by the interior of $D_I$ has the same image as $X$, that is $\rho(X)= \rho(\tilde X)$, then their cycles coincide.   
This means that the image of their reduction morphisms satisfy $ \varphi_i(\tilde X) \in  G \cdot \varphi_i(X)$. However, 
$G \subset SL(3, \mathbb C)$, which implies that $X \cong \tilde X$.   Therefore, $\rho$ is injective on the interior of $D_I$.  A straightforward iteration of this argument, using the fact that our dual graphs are always trees, applies to the deeper strata, and shows that $\rho$ is injective on $D_I$ itself.
\end{proof}


\bibliographystyle{alpha}
\bibliography{GoodLoci}


\end{document}